\documentclass[11pt]{article} 

\usepackage[utf8]{inputenc} 
\usepackage[T1]{fontenc}
\usepackage[T1]{CJKutf8} 
\usepackage[english]{babel}
\usepackage{amsmath, amsthm, amssymb}
\usepackage{amssymb,amsfonts,textcomp}
\usepackage{supertabular}
\usepackage{hhline}
\usepackage{multicol}
\usepackage{array}
\usepackage{graphicx} 
\usepackage{color}
\usepackage{hyperref}
\usepackage{placeins}
\usepackage{algorithmicx}
\usepackage[ruled]{algorithm}
\usepackage{algpseudocode}
\usepackage{appendix}
\usepackage{footnote}
\usepackage{subfigure}
\usepackage{caption}
\usepackage{color, colortbl}

\theoremstyle{plain}
\newtheorem{theorem}{Theorem}

\theoremstyle{definition}
\newtheorem{definition}{Definition}
\newtheorem{example}{Example}

\newtheorem{question}{Question}
\newtheorem{open problem}{Open problem}

\theoremstyle{remark}

\definecolor{Gray}{gray}{0.9}
\definecolor{LightCyan}{rgb}{0.88,1,1}
\definecolor{Red}{rgb}{1,0,0}
\definecolor{Orange}{rgb}{1,0.5,0}

\usepackage{geometry} 
\usepackage{fancyhdr} 
\usepackage{sectsty}
\allsectionsfont{\sffamily\mdseries\upshape} 

\def\@fnsymbol#1{\ensuremath{\ifcase#1\or *\or \dagger\or \ddagger\or
   \mathsection\or \mathparagraph\or \|\or **\or \dagger\dagger
   \or \ddagger\ddagger \else\@ctrerr\fi}}

\title{A new approach to catalog small graphs of high even girth }

\author{Vivek S. Nittoor\\
\small\tt vivek@nittoor.com\\
\small\tt Independent Consultant \& Researcher \\ \small\tt (formerly with the University of Tokyo)}
\date{} 

\begin{document}
\maketitle

\begin{abstract}
{A catalog of a class of $(3, g)$ graphs for even girth $g$ is introduced in this paper.  A $(k, g)$ graph is a regular graph with degree $k$ and girth $g$. 
This catalog of $(3, g)$ graphs for even girth $g$ satisfying $6 \le g \le 16$, has the following properties. Firstly, this catalog contains the smallest known $(3, g)$ graphs. An appropriate class of cubic graphs for this catalog has been identified, such that the $(3, g)$ graph of minimum order within the class is also the smallest known $(3, g)$ graph. Secondly, this catalog contains $(3, g)$ graphs for more orders than other listings. 
Thirdly, the class of graphs have been defined so that a practical algorithm to generate graphs can be created. Fourthly, this catalog is infinite, since the results are extended into knowledge about infinitely many graphs.\\
The findings are as follows. Firstly, Hamiltonian bipartite graphs have been identified as a promising class of cubic graphs that can lead to a catalog of $(3,g)$ graphs for even girth $g$ with graphs for more orders than other listings, that is also expected to contain a $(3, g)$ graph with minimum order. Secondly, this catalog of $(3,g)$ graphs contains many non-vertex-transitive graphs. Thirdly, in order to make the computation more tractable, and at the same time, to enable deeper analysis on the results, symmetry factor has been introduced as a measure of the extent of rotational symmetry along the identified Hamiltonian cycle. The D3 chord index notation is introduced as a concise notation for cubic Hamiltonian bipartite graphs. The D3 chord index notation is twice as compact as the LCF notation, which is known as a concise notation for cubic Hamiltonian graphs. The D3 chord index notation can specify an infinite family of graphs. 
Fourthly, results on the minimum order for existence of a $(3, g)$ Hamiltonian bipartite graph, and minimum value of symmetry factor for existence of a $(3, g)$ Hamiltonian bipartite graph are of wider interest from an extremal graph theory perspective. 
}
\end{abstract}


\section{Introduction}
A $(k, g)$ graph is a regular graph with degree $k$ and girth $g$. Cubic graphs are regular graphs with degree $3$. Only graphs that are undirected and simple are considered in this paper.\\
Tutte 1947 \cite{42} posed the cage problem as a problem in extremal graph theory, which refers to finding the smallest $(k, g)$ graph, which is referred to as the $(k, g)$ cage. The order of a $(k, g)$ case is referred to as $n(k, g)$. This problem has been solved for a limited range of degree $k$ and specified girth $g$. The $(3, g)$ cage is currently known only for $g  \le 12$.\\
The difficulty in finding the smallest $(3, g)$ graph in general is illustrated in the following historical example. The $(3, 14)$ vertex-transitive\footnote{Definition \ref{vt_def} in subsection \ref{subsec_sym}} graph with order $406$ was found by Hoare \cite{Hoare1983} in 1981, and it was only in 2002 that a smaller $(3, 14)$ graph with order $384$ was found by Exoo \cite{Exoo} outside the vertex-transitive class. \\
A catalog of $(3, g)$ graphs for even girth $g$ has been listed by the present author in \cite{CatalogPaper}. This catalog of $(3, g)$ Hamiltonian bipartite graphs (HBGs) has graphs for the range of girth $g$ satisfying $6 \le g \le 16$. The detailed comparison of this catalog with other lists and the infiniteness of this catalog is discussed in \cite{OverallPaper2}. A discussion on $(3, 14)$ graphs and the partial likelihood of the $(3, 14)$ record graph found by Exoo being a cage has been provided in \cite{3_14Paper}.\\ 
The goals for this research are listed in this section, the related works for this research are discussed in section \ref{sec_background} and some important observations are made in subsection \ref{observations}, which motivate the catalog in section \ref{sec_catalog}. The consequent findings arising from this work are listed in section \ref{sec_findings}, and limitations and open problems in section \ref{sec_conclusion_graph_analysis}.

\subsection{Context of this catalog}
The enumeration of cubic symmetric\footnote{Definition \ref{sym_def} in subsection \ref{subsec_sym} } graphs began with the Foster census and was expanded by Conder et al. \cite{mconder} up to order $10,000$. In $2012$, Potočnik et al. \cite{VTcen1}, \cite{VTcen2} extended the list to cubic vertex-transitive graphs up to order $1280$. Sophisticated techniques in these developments are based upon the classification of finite simple groups. These lists provide useful data to various areas of graph theory. \\
A more detailed literature review of catalogs or lists of graphs is provided in section \ref{sec_background}, that points out that the enumeration of all non-isomorphic cubic graphs for orders greater than 32 is currently open.

The goal of this research is to create a new catalog of $(3, g)$ graphs with graphs for many more orders than other known lists, with properties as follows.
\begin{enumerate}
\item \textbf{Goal 1}: This catalog should contain the smallest known $(3, g)$ graphs. As the class of vertex-transitive graphs does not attain this, this catalog should cover a class of graphs wider than vertex-transitive graphs. One task is to identify an appropriate class of graphs for this catalog, such that the $(3, g)$ graph of minimum order within the class is also the smallest known $(3, g)$ graph.
\item  \textbf{Goal 2}: This catalog should contain graphs with larger number of orders and larger girths than the existing listings, with graphs outside the vertex-transitive class for many orders. This will not be attained if all non-isomorphic graphs are enumerated. To overcome this hurdle, subclasses are defined, and list one representative from each subclass (if the subclass is not empty).
\item  \textbf{Goal 3}: The class of graphs for goal 1 and subclasses for goal 2 should be defined so that a practical algorithm to generate graphs can be created. This is challenging because of the difficulty of generating graphs of large girths. 
\item  \textbf{Goal 4}: The resulting graph catalog should be analyzed. As computational approach is obviously limited in that it can create finite number of results, an analysis should be developed that extends the results into knowledge about infinitely many graphs and knowledge in extremal graph theory perspective.
\end{enumerate}

\subsection{About Hamiltonicity and bipartiteness}
The choice of class of cubic graphs for this catalog \cite{CatalogPaper} is Hamiltonian bipartite. This is discussed further in section \ref{sec_app}. Some of the background behind making this choice is explained in this subsection.\\
The Lovasz conjecture 1969 \cite{Lovasz} states, ``Every finite connected vertex-transitive graph contains a Hamiltonian path''. A variant of this conjecture is as follows, ``Every finite connected vertex-transitive graph contains a Hamiltonian cycle, with the exception of the five known counterexamples.'' This conjecture is currently unsettled. All known vertex-transitive graphs are Hamiltonian, with the exception of the five counterexamples which are the Petersen graph, complete graph of order 2, $K_2$, the Coxeter graph, and two graphs derived from the Petersen graph and Coxeter graph by replacing each vertex with a triangle.\\
A recent study on non-Hamiltonian $3$-regular graphs of arbitrary girth by Haythorpe $2014$ \cite{Haythorpe2014} shows that the smallest known non-Hamiltonian $3$-regular graphs of girth $g$ are larger than the corresponding $(3, g)$ cages until $g=7$. 
The smallest order for a non-Hamiltonian $(3, 6)$ graph from Haythorpe $2014$ \cite{Haythorpe2014} is found to be 30, and the $(3, 6)$ cage has order 14.
Hence, it does appear that Hamiltonicity might be an important requirement for finding small graphs of large girth.\\
It is also well known that all the currently known $(3, g)$ cages for even girth $g$ are Hamiltonian. Open problem \ref{conj_chord_Hcg} is raised in Section \ref{chap_14_16_conclu} about whether every $(3, g)$ cage, other than the known exception of the $(3, 5)$ cage, the Petersen graph, has a Hamiltonian cycle.\\
All the currently known cages for even girth are bipartite. An still open problem raised by Harary 1970s, and Wong 1982 \cite{30}, ``Is every cage of even girth bipartite?''\\
When one considers the smallest non-bipartite vertex-transitive graph for even girth in Table \ref{Non-bipartite}, and when one compares it with Table \ref{table_sub_problems_cayley_vt}, it is clear that the empirical data suggests that the order of the smallest non-bipartite vertex-transtive graph of even girth is much greater than the smallest vertex-transtive graph of even girth.\\
Hence, Hamiltonian bipartite could be a very promising class of cubic graphs for the catalog of graphs of large even girth. 



\begin{table}
\centering
\caption{Order of smallest non-bipartite vertex-transitive graph for even girth from Potočnik et al. 2012 \cite{VTcen1} and \cite{VTcen2}}
\label{Non-bipartite} 
\begin{tabular}{cccccclllll}
\hline\noalign{\smallskip}
Girth $g$ & Order of smallest non-bipartite vertex-transitive graph  \\  
\noalign{\smallskip}
\hline
\noalign{\smallskip} 
6 & 24 \\
8 & 42 \\
10 & 100 \\
12 & 234\\
14 & 486\\
\hline
\end{tabular}
\end{table}

\subsection{About symmetry}
\label{subsec_sym}

It is well known that most graphs are asymmetric, that is, have no non-trivial automorphisms. The various kinds of graph symmetry defined in the literature have been based on the properties of the automorphism group of the graph, and some of them are defined in this subsection.
\begin{definition} \textbf{Vertex-transitive} Lauri J. et al. \cite{Sym_book} \\
\label{vt_def}
``We say that a graph $G$ is \textit{vertex-transitive} if any two vertices of $G$ are similar, that is, if, for any $u, v \in V(G)$, there is an automorphism $\alpha$ of $G$ such that $\alpha(u) = v$.'' 
\end{definition}
\begin{definition} \textbf{Edge-transitive} Lauri J. et al. \cite{Sym_book} \\
 ``A graph $G$ is \textit{edge-transitive} if, given any two edges $\{a, b\}$ and $\{c, d\}$, there exists an automorphism $\alpha$ such that $\alpha\{a, b\} = \{c, d\}$, that is, $\{\alpha(a),\alpha(b)\} = \{c, d\}$.''
 \end{definition}
\begin{definition}  \textbf{Semi-symmetric} \\
An \textit{edge-transitive} graph that is also regular but not \textit{vertex-transitive} is said to be \textit{semi-symmetric}. 
\end{definition}
The smallest \textit{semi-symmetric} graph is known to be the Gray graph of order 54, which is known to be cubic, Hamiltonian, and bipartite. 

\begin{definition} \textbf{Symmetric graph} \\
\label{sym_def}
A graph is said to be \textit{symmetric} if it is both vertex-transitive as well as edge-transitive. 
\end{definition}
\begin{definition} \textbf{Arc-transitive} Lauri J. et al. \cite{Sym_book}\\
``If $G$ has the property that, for any two edges $\{a, b\},\{c, d\}$, there is an automorphism $\alpha$ such that $\alpha(a) = c$ and $\alpha(b) = d$ and also an automorphism $\beta$ such that $\beta(a) = d$ and $\beta(b) = c$, then $G$ is said to be \textit{arc-transitive}.''
\end{definition}

\begin{definition} \textbf{Distance-transitive} Lauri J. et al. \cite{Sym_book}\\
``A graph $G$ is said to be \textit{distance-transitive} if given any vertices $a, b, c, d$ such that $\rho(a, b) = \rho(c, d)$, there is an automorphism $\alpha$ of $G$ such that $\alpha(a) = c$ and $\alpha(b) = d$.'' 
\end{definition}

Many of the known cages have a high level of symmetry. The $(3, 6)$ cage, the Heawood graph is known to be bipartite, Hamiltonian, distance-transitive, vertex-transitive, edge-transitive and hence symmetric. The $(3, 8)$ cage, the Tutte-Coxeter graph is known to be bipartite, Hamiltonian, distance-regular, distance-transitive, vertex-transitive, edge-transitive and hence symmetric. The $(3, 12)$ cage, the Tutte-12 cage or Benson graph is known to be bipartite, Hamiltonian, edge-transitive but not vertex-transitive and hence semi-symmetric.\\
The LCF notation was introduced by Lederberg \cite{Lederberg} and by Frucht \cite{LCFpaper}. The LCF notation is general in the sense that it can represent any arbitrary cubic Hamiltonian graph.\\
An interesting result from Rodriquez \cite{Rodriquez} is the following, `` A cubic graph $G$ is Hamiltonian if and only if it has a representation in LCF notation.''\\
The D3 chord index notation has been introduced in Definition \ref{notation_trivalent_bg_m} in Section \ref{sec_catalog}, has been motivated by circulant graphs and the LCF notation, but is meant to represent cubic Hamiltonian bipartite graphs in a more compact manner than the LCF notation. Symmetry factor as a measure of rotational symmetry as per chosen Hamiltonian cycle has been introduced in Definition \ref{def_sym_fac} in Section \ref{sec_catalog}, and is similar to the exponent in the LCF notation. Similar to the LCF notation, the D3 chord index notation becomes more compact for cubic Hamiltonian bipartite with rotational symmetry.

\section{Related Works}
\label{sec_background}

The focus of this section is a review of the literature on tables and catalogs of graphs.\\
Various researchers have commented on the importance of tables of graphs. Quoting from Read 1981 \cite{Read1981}: ``In this paper I shall give some information on recent advances in the generation of catalogues of graphs; but first it might be as well as to say a little about why one would want to generate such catalogues at all – why, for example, one would wish to produce all the graphs on 8 vertices.
There are many uses to which such a list could be put. Scrutiny of the list by hand or by computer, may suggest conjectures, or settle some question by turning up a counterexample. It may also enable one to get general ideas about graphs and their properties. Sometimes a list of graphs will supply numerical information for enumerative problems where a theoretical solution is absent, or provide a source of specimen graphs can be taken for use in one of the real-life problems to which graph theory can be applied''.\\
Faradzhev 1976 in \cite{Faradzev}, said that graph theory was in a botonical stage of development and a ``herbarium'' of graphs was a useful thing to have around.\\
A recent paper by Brinkmann et al. 2012 \cite{HoG} states, ``Such lists can serve as a source for intuition when one studies some conjecture and even as a possible source for counterexamples.''
The important survey papers found in the area of graph generation are Read 1981 \cite{Read1981} which focusses on simple graphs and Brinkmann et al. 2013 \cite{Brinkmann_hist2013} which focusses on cubic graphs.

\subsection{Research on enumeration of simple graphs}

An important research on counting of simple graphs is described in this subsection. The number of simple graphs for relatively small orders is known to be large. Harary et al. 1973 \cite{HararyPalmer} provide an asymptotic formula for number of simple graphs for a specified order. Read 1981 \cite{Read1981} lists early developments for listing graphs. Kagno 1946 \cite{Kagno1946} enumerated and showed that the number of simple graphs with order $6$ is $156$. Heap B.R. 1972 \cite{Heap1972}  enumerated and showed that the number of simple graphs with order $8$ is $12346$. Baker, Dewdney et al. 1974 \cite{Baker1974} enumerated and showed that the number of simple graphs with order $9$ is $274668$. The number of simple graphs with order $10$ is mentioned as $12005168$ by Harary et al. 1973 \cite{HararyPalmer}. 
The number of connected graphs with order $14$ is given as $29003487462848061$ in Brinkmann et al. \cite{HoG} 2013. The number for simple graphs for order 24 is 195704906302078447922174862416726256004122075267063365754368 as per Keith M. Briggs, Combinatorial Graph Theory \url{http://keithbriggs.info/cgt.html}.



\subsection{Research on enumerating cubic graphs}


\begin{table}
\caption{History of enumeration of cubic graphs from Read 1981 \cite{Read1981} and Brinkmann et al. 2013 \cite{Brinkmann_hist2013}}
\label{table_con_triv}
\centering
\begin{tabular}{cccccc}
\hline
Up to order & Author  & Year & Reference\\
\noalign{\smallskip}
\hline
\noalign{\smallskip}

10 & Vries  & 1889, 1891  & \cite{Vries1889},  \cite{Vries1891}\\
12 & Balaban & 1967  & \cite{Balaban1967}\\
14 & Bussemaker et al. & 1976 & \cite{Bussemaker1976}\\
18 & Faradzev & 1976& \cite{Faradzev} \\
20 & McKay et al. & 1986 & \cite{McKay1986} \\
24 & Brinkmann & 1992  & \cite{Brinkmann1992} \\
26 & Sanjmyatav et al. & 2000 & \cite{Sanjmyatav2000} \\
32 & Brinkmann et al. & 2011 & \cite{Brinkmann2011_gen}\\
\hline
\end{tabular}
\end{table}

Quoting from Meringer 1999 \cite{Meringer1999}, ``The construction of complete lists of regular graphs up to isomorphism is one of the oldest problems in constructive combinatorics.'' A recent survey paper on the history of generation of cubic graphs, Brinkmann et al. 2013 \cite{Brinkmann_hist2013} lists many of the early papers and key developments in this area. In the $19^{th}$ century, Vries  \cite{Vries1889},  \cite{Vries1891} enumerated all cubic graphs until order $10$. Balaban 1967 \cite{Balaban1967} enumerated all 3-regular graphs up to order 12. Bussemaker et al. 1976 \cite{Bussemaker1976} enumerated and showed that the number of cubic graphs with order up to $14$ is $504$.\\
The number of connected cubic graphs obtained from Robinson et al. 1983 \cite{Robinson1983} is 117 940 535 for order 24, and 8832736318937756165 for order 40.\\
The number of cubic bipartite graphs with order $24$ is given as $29579$ in a recent study by Brinkmann et al. 2013 \cite{Brinkmann2013}. Quoting from \cite{Brinkmann2013}, "Cubic graphs (alias trivalent graphs) constitute an important family of graphs that are worth enumerating and generating."\\ 
Table \ref{table_con_triv} shows important papers in the history of enumeration of cubic graphs until order 32, and enumeration of cubic graphs for orders greater than 32 is currently unresolved.

\subsection{Enumeration of classes of cubic graphs}
The enumeration of cubic symmetric graphs that began with the Foster Census \cite{FosterC} up to order $512$ and recently generalized to order $10000$ by Conder et al. \cite{mconder} and available at \url{https://www.math.auckland.ac.nz/~conder/symmcubic10000list.txt} is an important application of the classification of finite, simple groups. The enumeration of cubic vertex-transitive graphs by Potočnik et al. \cite{VTcen1}, \cite{VTcen2} in $2012$ is a logical extension of the enumeration of cubic symmetric graphs. Cubic symmetric graphs are a subset of cubic vertex-transitive graphs. These enumerations are important from a wider graph theory perspective, because they provide new knowledge about the class of graphs being enumerated and also provide useful data of interest to other problems from graph theory. For example, the enumerations of cubic vertex-transitive graphs have been used to confirm the order of smallest $(3, g)$ vertex-transitive graph until $g = 16$ in \cite{VTcen1}.\\
The number of non-isomorphic cubic vertex-transitive graphs of order of at most 1280 is mentioned as 111360 by Potočnik et al. in 2012.

\subsection{Extremal graph theory and cage problem}
The subject of extremal graph theory originated in the $1940$s due to Turan and Erdos. Extremal graph theory concerns itself with various invariant properties of a graph. Given any property for a class of graphs, an extremal graph theory question would be to determine the smallest value of a graph invariant such that it satisfies the specified property.\\
The known lower and upper bounds for the cubic cage problem are extremely wide for larger values of girth $g$, and hence finding better bounds and more cages is a challenging problem. The latest list of known cages is available at \cite{48}. 
Computational methods have been used to improve lower bounds of the cage problem by Brinkmann et al. \cite{Brinkmann} in $1995$ for $(3, 9)$ graphs, McKay et al. \cite{105} in $1998$ for $(3, 11)$ graphs, and by Exoo et al. \cite{112} in $2011$ for $(4, 7)$ graphs by showing the non-existence of $(4, 7)$ graphs with less than order $67$, and finding a $(4, 7)$ graph with order $67$. The lower bound for $(3, 14)$ graphs was improved from $256$ to $258$ by McKay et al. \cite{105} in $1998$.

\subsection{Other lists of graphs}

The important lists of graphs that we find on the internet are as follows.
\begin{enumerate}
\item \url{http://cs.anu.edu.au/~bdm/data/} 
Some of the enumerations of graphs are as follows.
\begin{itemize}
\item Eulerian graphs.
\item Strongly regular graphs.
\item Ramsey graphs. A survey can be found in Radziszowski \cite{Radziszowski} 2014.
\item Hypohamiltonian graphs.
\item Planar graphs.
\item Semiregular bipartite graphs.
\item Self-complementary graphs.
\item Highly irregular graphs.
\item Digraphs: Tournaments, Regular and Semi-regular Tournaments, Locally-transitive Tournaments, Acyclic digraphs.
\item Multigraphs 
\end{itemize}
\item \url{http://www.mathe2.uni-bayreuth.de/markus/reggraphs.html}
Meringer \cite{Meringer1999} 1999 lists many kinds of graphs. 
\begin{itemize}
\item Connected regular graphs with girth at least 6: The number of connected regular graphs with girth at least 6 with order 30 is given as 122090544.  
\item Connected regular graphs with girth at least 4.
\item Connected regular graphs with girth at least 5.
\item Connected regular graphs with girth at least 6.
\item Connected regular graphs with girth at least 7.
\item Connected regular graphs with girth at least 8.
\item Connected bipartite regular graphs.
\item Connected planar regular graphs.
\end{itemize}
\item \url{http://www.maths.gla.ac.uk/~es/}
Strongly regular graphs with order of at most 64 have been listed by Spence. The relevant papers are Spence \cite{Spence1} 2000, Spence \cite{Spence2} 1995, McKay et al. \cite{Spence3} 2001, Haemers et al. \cite{Spence4} 2001, and Coolsaet et al. \cite{Spence5} 2006.
\item House of Graphs  \url{http://hog.grinvin.org} 
Brinkmann et al. \cite{HoG} 2013 lists graphs from many different sources mentioned above in this survey, but this list does not contain graphs of high girth. For example, a search for graphs with girth greater than or equal to 12 yields only one graph.
\item Coxeter et al. 1981 \cite{ZS_census} lists a census of zero-symmetric graphs\footnote{A zero-symmetric graph is a cubic vertex-transitive graph with an automorphism group that partitions its edges into three orbits} until order 120. The catalog of cubic vertex-transitive graphs of Potočnik et al. 2012 \cite{VTcen1}, \cite{VTcen2} extends the census of zero-symmetric graphs until order 1280.
\item Conder at al. 2006 \cite{Conder_semi} listed a census of cubic semi-symmetric graphs until order 768.

\end{enumerate}
The only sources for graphs of high girth are literature of the cage problem and enumerations of cubic vertex-transitive and symmetric graphs.

\subsection{Other computational approaches to find graphs}

The two researches that have some similarity to our work are McKay et al. \cite{105} 1997 and Meringer \cite{Meringer1999} 1999. The key difference is that isomorphism rejection is not used in our approach, and we focus on Hamiltonian bipartite class of cubic graphs.

\subsection{Other related researches}
Holton et al. \cite{Holton85} 1985 show that 3-connected cubic bipartite planar graphs with orders lesser than $66$ are Hamiltonian. After removing the restriction on bipartite graphs, Holton et al. \cite{Holton88} 1988 show that the smallest non-Hamiltonian 3-connected cubic planar graphs has order $38$.
Other studies on regular graphs of high girth are as follows.
\begin{enumerate}
\item McKay \cite{McKay1987} 1987 studied properties of independent sets on regular graphs of high girth.
\item Goldberg \cite{Goldberg1992} 1992 discusses the problem of designing an efficient algorithm for listing all graphs with order $n$, listing one from each isomorphism class of graphs with order $n$.
\end{enumerate}

\subsection{Observations}
\label{observations}
Some observations from the literature review mentioned earlier in this section are made here as follows. 

\begin{enumerate}
\item \textbf{Difficulty of finding cubic graphs of large girth}\\
As pointed out earlier, there is an inherent difficulty in finding cubic graphs of large girth outside the vertex-transitive class. 
\item \textbf{Large number of non-isomorphic cubic graphs for orders greater than 32}\\ 
As we have seen in this section, 
in general there are a large number of non-isomorphic cubic graphs for a specified order, thus making it practically very difficult to list all of them for orders greater than 32. 
\item \textbf{Beyond the class of cubic vertex-transitive graphs}\\
The main literature for cubic graphs of high girth are that of the cage problem and enumerations of cubic vertex-transitive and symmetric graphs. Given the list of cubic vertex-transitive graphs, we can find the $(3, g)$ vertex-transitive graph with the minimum order in the list. The minimum orders for $(3, g)$ Cayley and $(3, g)$ vertex-transitive graphs are obtained from Potočnik et al. \cite{VTcen1} and \cite{VTcen2} 2012, and the lower and upper bounds from Exoo et al. 2011 \cite{Jajcaysurvey} and are shown in Table \ref{table_sub_problems_cayley_vt} for even girth $g$. 
There are two main limitations of currently known listings of $(3, g)$ graphs. Firstly, there is currently no other known listing of $(3,g)$ graphs, where the minimum $(3,g)$ graph within the class is also the smallest known $(3, g)$ graph for $g > 8$. Secondly, there do not exist $(3, g)$ vertex-transitive graphs for many orders. For example, there does not exist a $(3, 10)$ vertex-transitive graph with order less than 80. The smallest three orders for which $(3, 14)$ vertex-transitive graphs exist, as mentioned by Potočnik et al. 2012  \cite{VTcen1} and \cite{VTcen2} are 406, 448 and 480. 
\end{enumerate}

\begin{table}
\centering
\caption{Values of minimum order of Cayley and vertex-transitive graphs for even values of $g \ge 6$ obtained from \cite{VTcen1} and \cite{VTcen2}, and lower and upper bounds obtained from \cite{Jajcaysurvey}}
\label{table_sub_problems_cayley_vt} 
\begin{tabular}{cccccclllll}
\hline\noalign{\smallskip}
Girth $g$ & $n_\mathit{Cayley}$ & $n_\mathit{Vertex\_transitive}$ &  $(3, g)$ Lower Bound &  $(3, g)$ Upper Bound \\  
\noalign{\smallskip}
\hline
\noalign{\smallskip} 
6 & 14 & 14 & 14 & 14\\
8 & 30 & 30 & 30 & 30\\
10 & 96 & 80 & 70 & 70\\
12 & 162 & 162 & 126 & 126\\
14 & 406 & 406 & 258 & 384\\
16 & 1008 & 1008 & 512 & 960\\
\hline
\end{tabular}
\end{table}



\section{Catalog}
\label{sec_catalog}
\label{sec_app}

In this section, the approach to graph listing is explained, with the choice of Hamiltonian bipartite class of cubic graphs and focus on efficient representation as a first step towards listing $(3, g)$ graphs within the identified class. 
%
%
The important steps in the approach to find graphs of high girth can be described as follows.
\begin{enumerate}
\item The search space for computer search is restricted to Hamiltonian bipartite class of cubic graphs.
\item An efficient representation for Hamiltonian cubic bipartite graphs with a specified level of rotational symmetry.
\item  A range of rotational symmetries are chosen for each value of $g$ such that a $(3, g)$ HBG with that level of rotational symmetry could be found by computer search.
\item We treat $(3, g)$ graphs of a particular level of rotational symmetry within the identified class of cubic graphs as a subclass, and seek to list at least one representative from each subclass.
\end{enumerate}

\subsection{Approach}
\label{sec_summary_approach_key_contributions}
\label{sec_comp_cage_problem}

The number of Hamiltonian cubic bipartite graphs are too numerous in order to practically enumerate them. Thus, as a first step, we try to list $(3, g)$ HBGs for each order for which $(3, g)$ HBGs exist until $2^{3g/4}$. The number $2^{3g/4}$ is motivated by the following open problem, ``Finding an infinite family of cubic graphs with large girth $g$ and order $2^{cg}$ for $c < 3/4$.'' from Exoo et al. \cite{Jajcaysurvey} 2011.
A computational approach for finding and listing $(3, g)$ HBGs for even girth $g$ for a range of orders until $2^{3g/4}$ and hence find a $(3, g)$ HBG with the minimum order, has been introduced. 

Hamiltonian bipartite have been identified as a promising class of cubic graphs for obtaining $(3, g)$ graphs of high even girth $g$. Every $(3, g)$ HBG of minimum order is also a $(3, g)$ cage for even girth $g \le 12$. Hamiltonian cubic bipartite graphs represent one quadrant in the space of cubic graphs as shown in Figure \ref{graphdiagramoverall}. The other three quadrants are Hamiltonian non-bipartite graphs, non-Hamiltonian bipartite graphs and non-Hamiltonian non-bipartite graphs. There are no known enumeration methods for cubic graphs outside the set of cubic vertex-transitive graphs other than this research.

\begin{figure}[htbp]
\centering
\includegraphics[scale=0.96]{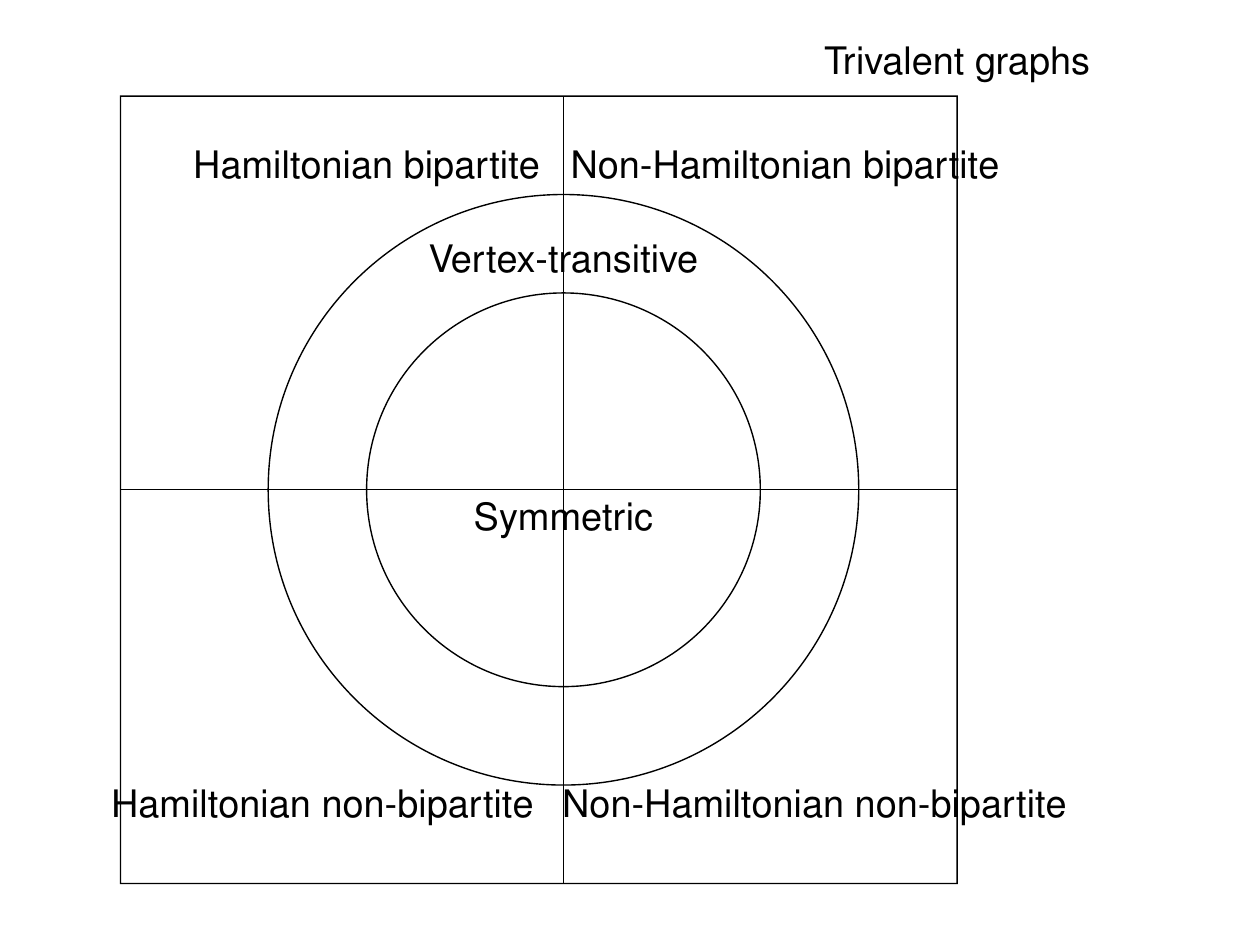}
\caption{Cubic graphs}
\label{graphdiagramoverall}
\end{figure}

The traditional method to represent a labeled graph is by its adjacency list. The LCF notation for representing Hamiltonian cubic graphs was introduced by Lederberg in \cite{Lederberg} and by Frucht in \cite{LCFpaper}.\\
The motivation for a compact notation for Hamiltonian cubic bipartite graph is as follows.
\begin{enumerate}
\item To represent any Hamiltonian cubic bipartite graph using fewer variables compared to the entire adjacency list. 
\item To represent any Hamiltonian cubic bipartite graph using fewer variables compared to the LCF notation.
\end{enumerate}

Even though the LCF notation is extremely convenient for representing a cubic HBG, it is not a minimal representation, and one of the motivations here is to propose a more compact graph representation format that would be practically useful in order to represent HBGs and also find HBGs by computer search. The D3 chord index notation for representing Hamiltonian cubic bipartite graphs, and a parameter called symmetry factor $b$ for a Hamiltonian cubic graph with order $2m$, where $b | m$ that reflects the extent of rotational symmetry are introduced as part of this work. As pointed out earlier, symmetry factor is very similar to the concept of exponent used in the LCF notation.



In general, any HBG of order $2m$ can be completely specified by the order and length of chords connected from alternate nodes. The concept of D3 chord index is introduced as length of chords connected from alternate nodes of a HBG and is formally introduced in Definition \ref{notation_trivalent_bg_m}. Thus, it is clear that any HBG has a D3 chord index representation. Further, by considering rotational symmetry along the Hamiltonian cycle, the concept of symmetry factor is introduced in Definition \ref{def_sym_fac_gen}.\\  
For a HBG of order $2m$ with rotational symmetry along Hamiltonian cycle, the D3 chord indices would typically look something like $d_1, d_2, ..., d_m$.
Hence, the same HBG can be more compactly represented by D3 chord indices $d_1, d_2, ..., d_b$, symmetry factor $b$ and order $2m$, assuming that $b$ divides $m$.\\
For other kind of symmetries such as folding symmetry, we would in general need to use the full symmetry factor representation.


\begin{definition} \textbf{D3 chord index notation} \\
\label{notation_trivalent_bg_m}
The D3 chord indices $l_{1}, l_{2}, \ldots, l_{m}$ for order $2m$ where each $l_{i}$ is an odd integer satisfying $3 \le l_{i} \le 2m - 3$ for $1\le i\le m$ is a labeled graph with order $2m$, with labels $1, 2, 3, \ldots, 2m$ constructed as follows. 
\begin{enumerate}
\item Vertex $1$ is connected to vertex $2m$,  vertex $2$ and vertex $1 + l_{1}$.
\item For integers $i$ satisfying $2 \le i \le m$, 
\begin{itemize}
\item Vertex $2i - 1$ is connected to vertex $2i - 2$.
\item Vertex $2i - 1$ is connected to vertex $2i$.
\item Vertex $2i - 1$ is connected to vertex $y_{i}$.
\end{itemize}
where $y_{i}$ is calculated as follows.
\begin{itemize}
\item $y_{i} = 2i - 1 + l_{i}$ if $2i - 1 + l_{i} \le 2m$.
\item $y_{i} = 2i - 1 + l_{i} \bmod 2m$ if $2i - 1 + l_{i} > 2m$.
\end{itemize}
\end{enumerate}
\end{definition}


\begin{definition}  \textbf{Symmetry factor for Hamiltonian cubic bipartite graph}\\
\label{def_sym_fac_gen}
\label{def_sym_fac}
A Hamiltonian cubic bipartite graph with order $2m$ is said to have symmetry factor $b \in \mathbb{N}$ if the following conditions are satisfied.
\begin{enumerate}
\item $b$ divides $m$. 
\item There exists a labelling of the vertices of the Hamiltonian cubic graph with order $2m$ have labels $1, 2, \ldots, 2m$, such that $1 \to 2 \to \ldots \to 2m \to 1$ is a Hamiltonian cycle that satisfy the following properties.
\begin{itemize}
\item The edges that are not part of the above Hamiltonian cycle are connected as follows.
Vertex $i$ is connected to vertex $u_{i}$ for $1 \le i \le 2m$.
\item If $j \equiv i \bmod 2b$ for $1 \le j \le 2b$ and $1 \le i \le 2m$ then the following is true, $u_{i} - i  \equiv u_{j} - j \bmod 2m$.
\end{itemize}
\end{enumerate}
\end{definition}


Figure \ref{graph_70_sym7_g10} shows an example of a $(3, 10)$ Cage with order $70$ and symmetry factor $7$.

\begin{figure}[htbp]
\centering
\includegraphics[scale=0.66]{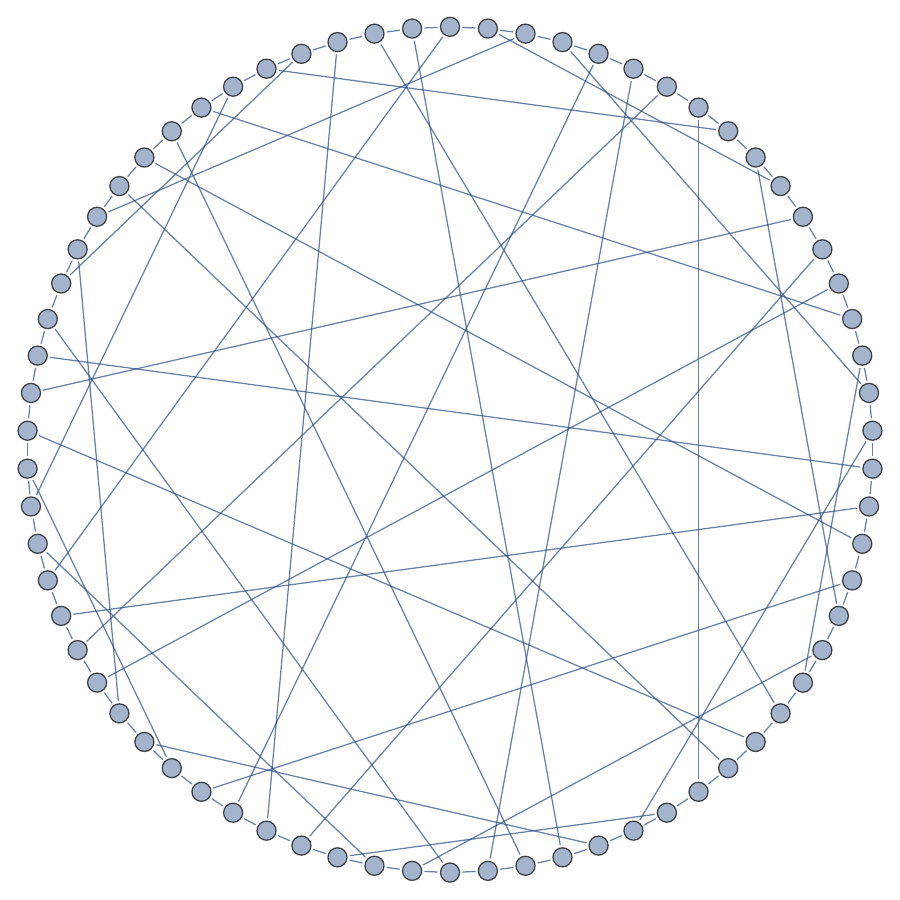}
\caption{\hspace{2em}  $(3, 10)$ Cage with order $70$ and symmetry factor $7$ found by Balaban 1972 \cite{32}}
\label{graph_70_sym7_g10}
\end{figure}

\begin{theorem} {\label{thm_sf_g}}
If a Hamiltonian cubic bipartite graph with symmetry factor $b$ has girth $g$ then $b \ge g/4 -1/2$.
\end{theorem}
\begin{proof}
Let us visualize a cycle of length $2m$ with vertices numbered $1, 2, \ldots, 2m$.
If vertex $1$ is connected to vertex $j$ and vertex $2b + 1$ to vertex $2b + j$, then we have a cycle of length $4b +2$.
Since $4b +2 \ge g$ we obtain $b\ge g/4-1/2$ as required.
\end{proof}




%
\label{notation_D3}
The notation D3 is used to refer to D3 chord indices $l_{1}, l_{2}, \ldots, l_{b}$ for a Hamiltonian cubic bipartite graph with symmetry factor $b$ and order $2m$ where $b | m$.

The D3 chord indices for the $(3, 10)$ cage shown in Figure \ref{graph_70_sym7_g10} are 9\ 13\ 29\ 21\ 13\ 43\ 33 with symmetry factor 7. The LCF representation for the same graph is [-29,\ -19,\ -13,\ 13,\ 21,\ -27,\ 27,\ 33,\ -13,\ 13,\ 19,\ -21,\ -33,\ 29]$^{5}$. The D3 chord index notation is twice as compact as the LCF notation for Hamiltonian cubic bipartite graphs. 


\subsection{Summary of catalog}

This catalog has been obtained by computer search for $(3,g)$ HBG by enumerating the space of D3 chord indices for a specified order and symmetry factor. The current approach is to exit the search when a representative $(3,g)$ HBG is found for a particular order and symmetry factor. A summary of the catalog has been provided in Table \ref{table_sub_problems}. Symmetry factor allows the decomposition of the problem of listing $(3, g)$ HBGs to sub-problems of listing $(3, g)$ HBGs for a range of symmetry factors. This catalog of $(3, g)$ HBG is infinite, and a more detailed discussion is provided in \cite{OverallPaper2}.\\
This computational approach to find a $(3, g)$ HBG with a particular symmetry factor $b$, and specified order $2m$, can conclude that such as $(3, g)$ HBG does not exist, or at times even be \textit{inconclusive} on whether such a $(3, g)$ HBG exists or not.\\
The listing of $(3, g)$ HBGs for a particular even value of girth $g$ until $2^{3g/4}$ is considered to be \textit{exhaustive} if all orders in specified range that have a $(3, g)$ HBG are listed, with proof for non-existence for orders not listed.\\  
The outcome for listing of $(3, g)$ HBGs is \textit{partial} if results on existence $(3, g)$ HBG for some orders in specified range until $2^{3g/4}$ are \textit{inconclusive}.
A listing of $(3, g)$ HBGs for a particular symmetry factor $b$ is considered to be monotonic, when there exists an order $2m$ for which the smallest $(3, g)$ HBG with symmetry factor $b$ exists, such that there exist at least one $(3, g)$ HBG with symmetry factor $b$ exists for all orders $2m + 2bk$, where $k$ is a natural number. If a listing of $(3, g)$ HBGs for a particular symmetry factor $b$ is not monotonic, it is considered to be non-monotonic.

In Table \ref{table_sub_problems}, the smallest $(3, g)$ HBGs for various symmetry factors have been listed.

\begin{table}
\centering
\caption{: Finding $(3, g)$ HBGs of minimum order with symmetry factor $b$}
\label{table_sub_problems} 
\label{table_sub_problems_6_16}
\begin{tabular}{lclllll}
\hline\noalign{\smallskip}
$b$ & $(3, 6)$ & $(3, 8)$& $(3, 10)$ & $(3, 12)$ & $(3, 14)$  &  $(3, 16)$  \\  
\noalign{\smallskip}
\hline
\noalign{\smallskip} 

1 & \cellcolor{Gray}14, Cage &\cellcolor{Red}  &\cellcolor{Red}  &\cellcolor{Red}  &\cellcolor{Red}  & \cellcolor{Red}  \\
2 & &\cellcolor{Gray}36  &\cellcolor{Gray}80     &\cellcolor{Red} &\cellcolor{Red}  &\cellcolor{Red}    \\
3 & &\cellcolor{Gray}30, Cage &\cellcolor{Gray}90  &\cellcolor{Gray}162&\cellcolor{Red} &\cellcolor{Red}\\
4 & &\cellcolor{Gray}40 &\cellcolor{Gray}72 &\cellcolor{Gray}216 &\cellcolor{Gray}440\\
5 & &\cellcolor{Gray}40 &\cellcolor{Gray} 80&\cellcolor{Gray}190&\cellcolor{Gray}460 \\
6 & &\cellcolor{Gray} 36&\cellcolor{Gray} 84&\cellcolor{Gray}168&\cellcolor{Gray}456 \\
7 & \cellcolor{Gray}14, Cage&\cellcolor{Gray}42 &\cellcolor{Gray}70, Cage  &\cellcolor{Gray}182& \cellcolor{LightCyan}406$^{\#}$ \\
8 & &\cellcolor{Gray}48 &\cellcolor{Gray} 80&\cellcolor{LightCyan}208$^{\#}$ &\cellcolor{LightCyan}384$^{\#}$, Record & $\cellcolor{LightCyan}1568^{\#}$\\
9 & &\cellcolor{Gray} 54&\cellcolor{Gray} 90&\cellcolor{Gray}126, Cage& \cellcolor{LightCyan}504$^{\#}$  \\
10 & &\cellcolor{Gray}40 &\cellcolor{Gray}80 &\cellcolor{LightCyan}200$^{\#}$ &\cellcolor{LightCyan}520$^{\#}$  \\
11 & &\cellcolor{Gray}44  &\cellcolor{Gray}88&\cellcolor{Orange}  &\cellcolor{LightCyan} 506$^{\#}$  \\
12 & &\cellcolor{Gray}48 &\cellcolor{Gray}72&\cellcolor{LightCyan}216$^{\#}$  &\cellcolor{LightCyan}576$^{\#}$ \\
13 & &\cellcolor{Gray}52 & & &\cellcolor{LightCyan}572$^{\#}$ \\
14 & &\cellcolor{Gray}56 & & &\cellcolor{LightCyan}588$^{\#}$  \\
15 & &\cellcolor{Gray}30, Cage & & &\cellcolor{LightCyan}600$^{\#}$ \\
16 & &\cellcolor{Gray}64 &  & &\cellcolor{LightCyan}576$^{\#}$\\
35 & && \cellcolor{Orange} & \\
\hline
\end{tabular}
\begin{tabular}{lllll}
\hline\noalign{\smallskip}
Color & Significance  \\  
\noalign{\smallskip}
\hline
\noalign{\smallskip} 
\label{table_sub_problems_colors7} 
 \cellcolor{Gray} & $(3, g)$ sub-problem resolved between known lower and upper bound\\
\cellcolor{LightCyan}$^{\#}$ & Upper bound $(3, g)$ sub-problem for HBGs 
 found, but open to improvement\\
\cellcolor{Red} & Upper bound $(3, g)$ HBG sub-problem resolved 
 - Does Not  Exist\\
\cellcolor{Orange} & $(3, g)$ HBG sub-problem is \textit{inconclusive}\\
\hline
\end{tabular}
\end{table}

The lower bound for $(3, 14)$ is known to be 258 as per \cite{Jajcaysurvey}. The lower bound for a $(3, 14)$ HBG for a particular symmetry factor $b$, would be the smallest positive integer greater than or equal to 258 that is also divisible by $2b$, which we denote as $lb(3, 14, b)$. The lower bounds for some symmetry factors that have improved as part of this research are shown in Table \ref{table_comp314}.

\begin{table}
\centering
\caption{$(3, 14)$ sub-problems lower and upper bounds from \cite{3_14Paper}}
\label{table_comp314}
\begin{tabular}{cllllll}
\hline
Symmetry & $lb(3, 14, b)$  & Lower bound $(3, 14)$ HBG & Upper bound $(3, 14)$ HBG\\
factor $b$ &  &  for symmetry factor factor $b$ & for symmetry factor $b$ \\
\noalign{\smallskip}
\hline
\noalign{\smallskip}

\cellcolor{Red}3 &\cellcolor{Red}258 & \cellcolor{Red}900  & \cellcolor{Red}\\
\cellcolor{Gray}4 &\cellcolor{Gray}264 &\cellcolor{Gray}440 & \cellcolor{Gray}440 \\
\cellcolor{Gray}5 &\cellcolor{Gray}260& \cellcolor{Gray}460 & \cellcolor{Gray}460 \\
\cellcolor{Gray}6  &\cellcolor{Gray}264 & \cellcolor{Gray}456 & \cellcolor{Gray}456 \\
\cellcolor{Orange}7 &\cellcolor{Orange}266 &\cellcolor{Orange}364 & \cellcolor{Orange}406 \\
\cellcolor{Orange}8  &\cellcolor{Orange}272 &\cellcolor{Orange}304& \cellcolor{Orange}\textcolor{blue}{384} \\
\cellcolor{Orange}9 &\cellcolor{Orange}270 &\cellcolor{Orange}288  & \cellcolor{Orange}504 \\
\cellcolor{LightCyan}10 &\cellcolor{LightCyan}260 &\cellcolor{LightCyan}260 & \cellcolor{LightCyan}460  \\
\cellcolor{LightCyan}11 &\cellcolor{LightCyan}264 & \cellcolor{LightCyan}264 & \cellcolor{LightCyan}506 \\
\cellcolor{LightCyan}12 &\cellcolor{LightCyan}264& \cellcolor{LightCyan}264 & \cellcolor{LightCyan}456 \\
\cellcolor{LightCyan}13 &\cellcolor{LightCyan}260 & \cellcolor{LightCyan}260 & \cellcolor{LightCyan}572 \\
\cellcolor{LightCyan}14 &\cellcolor{LightCyan}280 & \cellcolor{LightCyan}280 & \cellcolor{LightCyan}588 \\
\cellcolor{LightCyan}15 &\cellcolor{LightCyan}270 & \cellcolor{LightCyan}270 & \cellcolor{LightCyan}510 \\
\cellcolor{LightCyan}16 &\cellcolor{LightCyan}288 & \cellcolor{LightCyan}288 & \cellcolor{LightCyan}\textcolor{blue}{384} \\
\hline
\end{tabular}
\begin{tabular}{lclllll}
\hline\noalign{\smallskip}
Color & Significance  \\  
\noalign{\smallskip}
\hline
\noalign{\smallskip} 
\label{table_sub_problems_colors} 
 \cellcolor{Red} & Found to not exist \\
\cellcolor{Gray} & Lower bound equals upper bound \\
\cellcolor{Orange} & Lower bound improved over $lb$\\
\cellcolor{LightCyan} & Scope for potentially improving bounds  \\
\hline
\end{tabular}\\
\textcolor{blue}{384} refers to the $(3, 14)$ Record graph \\
\end{table}

An example from this catalog of $(3, g)$ graphs is provided in Figure  \ref{ex_inf}, which shows that $(3, 6)$ HBGs exist for all even orders greater than or equal to 14. An example from this catalog of $(3, g)$ graphs is provided in Figure  \ref{ex_inf}, which shows that $(3, 8)$ HBGs with symmetry factor $4$ exist for all even orders greater than or equal to 40. Hence, both these examples can be considered monotonic. \\ 
A non-montonic example from this catalog of $(3, g)$ graphs is provided in Figure  \ref{ex_mono_not}, which shows that $(3, 8)$ HBGs with symmetry factor $3$ exist for all even orders greater than or equal to 42, there does not exist a $(3, 8)$ HBGs with symmetry factor $3$ with order 36, even though the smallest $(3, 8)$ HBGs with symmetry factor $3$ is of order 30 and is the $(3, 8)$ cage.

\begin{figure}[htpb]
\centering
\caption{$(3, 6)$ HBGs exist for all even orders greater than or equal to 14.}
\label{ex_inf}  
\label{example_3_6_inf}
\begin{tabular}{l | l | l | l | l | l | l | c}
\noalign{\smallskip}
\hline
\noalign{\smallskip}  
Order & 10 & 12 & 14 & 16 & 18 & $20 + 2i,$ \\
 &  &  &  &  &  & $i = 0, 1, \ldots$ \\
\hline
\noalign{\smallskip}  
Existence& $\neg$ $\exists$ & $\neg$ $\exists$ & $\exists$  & $\exists$  & $\exists$  & $\exists$  \\
\hline
\noalign{\smallskip}  
Figure & & & {\includegraphics[scale=0.1]{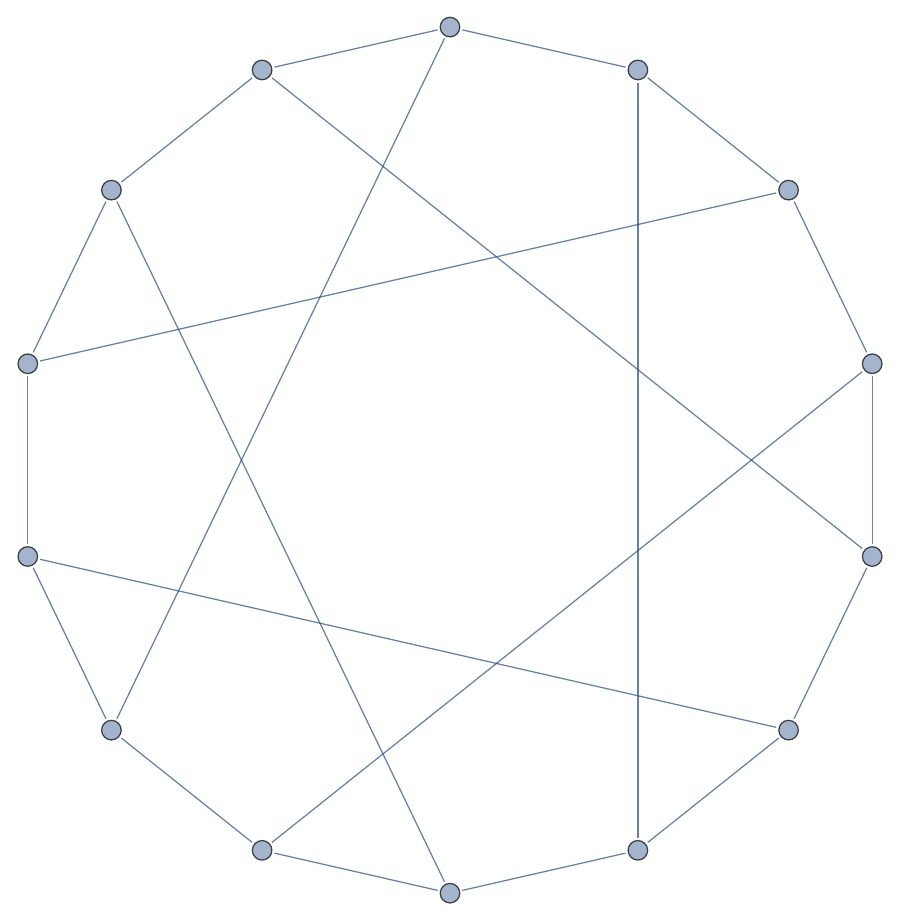}}  & {\includegraphics[scale=0.1]{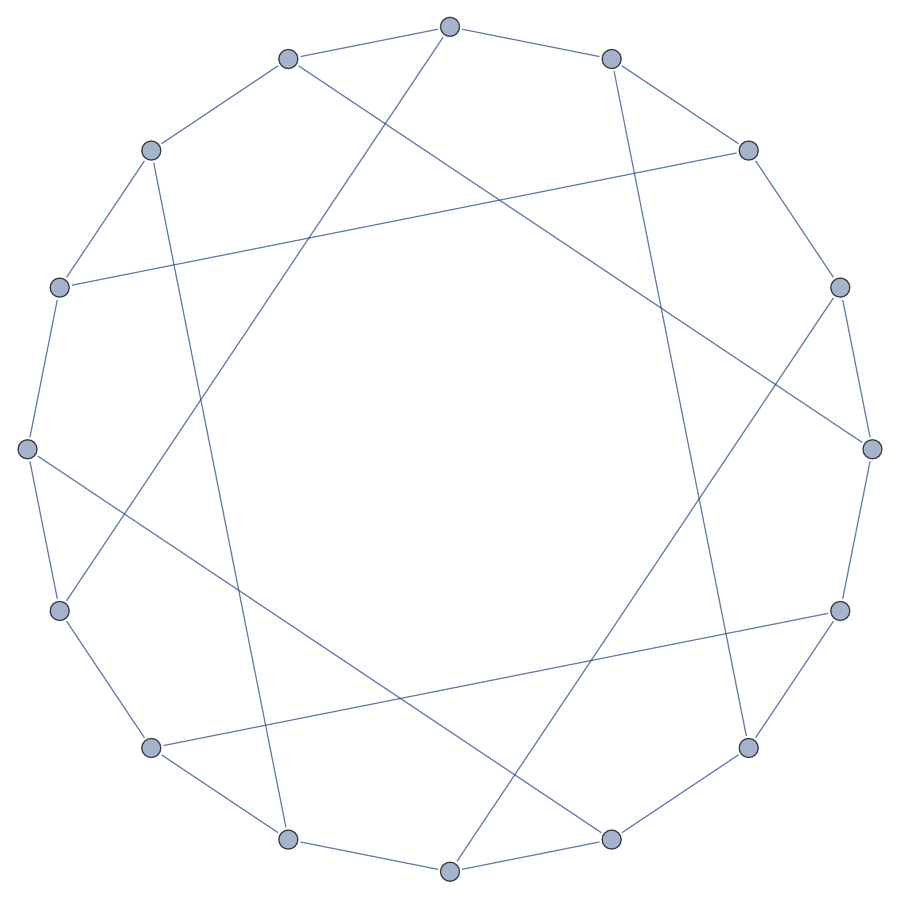}} & {\includegraphics[scale=0.1]{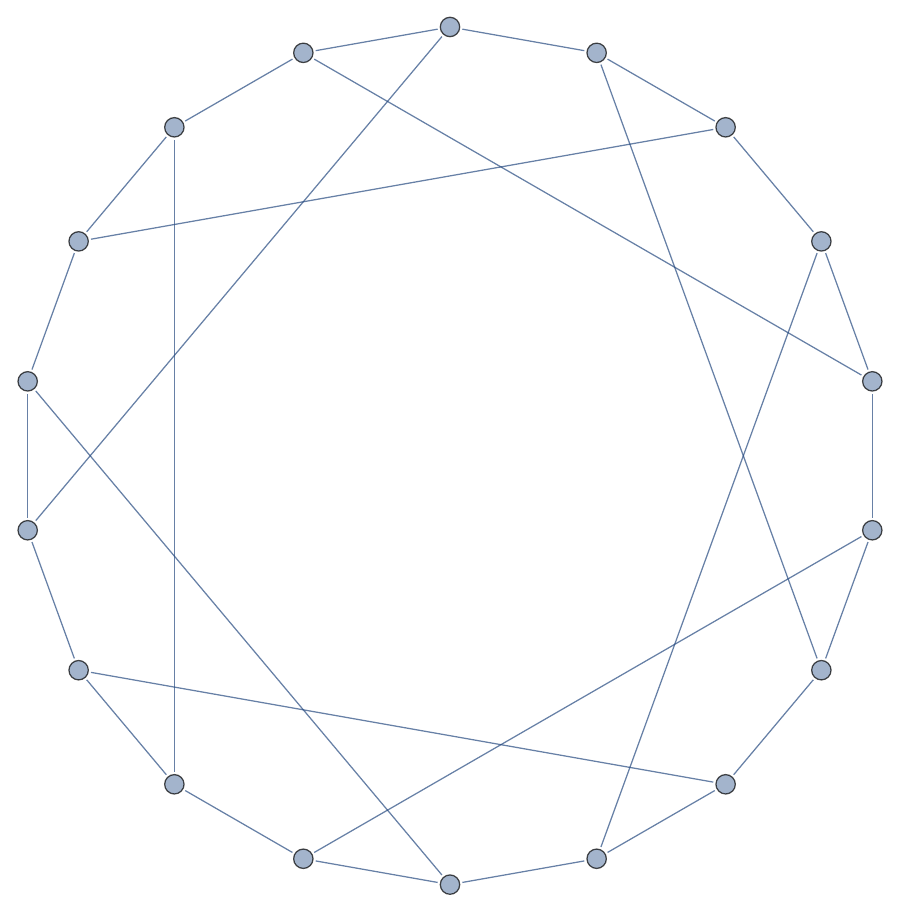}} & $\ldots$ \\
\hline
\noalign{\smallskip}  
D3   &   &   & 5 & 5 & 5 & 5  \\
& 
 \\
\hline
\end{tabular}
\end{figure}

\begin{figure}[htpb]
\centering
\caption{$(3, 8)$ HBGs with symmetry factor $4$ (\textbf{monotonic})} 
 \label{ex_mono} 
\label{table_graph_g8_sym4}
\begin{tabular}{l | l | l | l | l | l | l | c}
\noalign{\smallskip}
\hline
\noalign{\smallskip}  
Order & 32 & 40 & 48 & 56 & 64 & $72 + 8i,$ \\
 &  &  &  &  &  & $i = 0, 1, \ldots$ \\
\hline
\noalign{\smallskip}  
Existence& $\neg$ $\exists$ & $\exists$ & $\exists$  & $\exists$  & $\exists$  & $\exists$  \\
\hline
\noalign{\smallskip}  
Figure &  & {\includegraphics[scale=0.1]{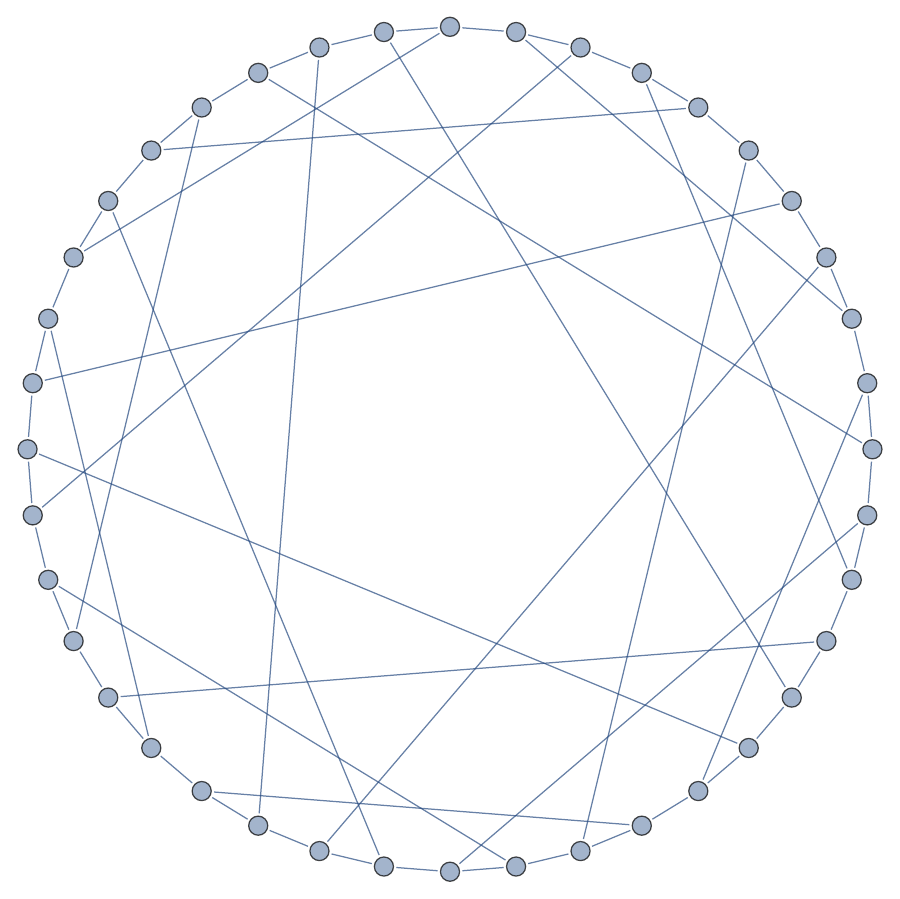}}  & {\includegraphics[scale=0.1]{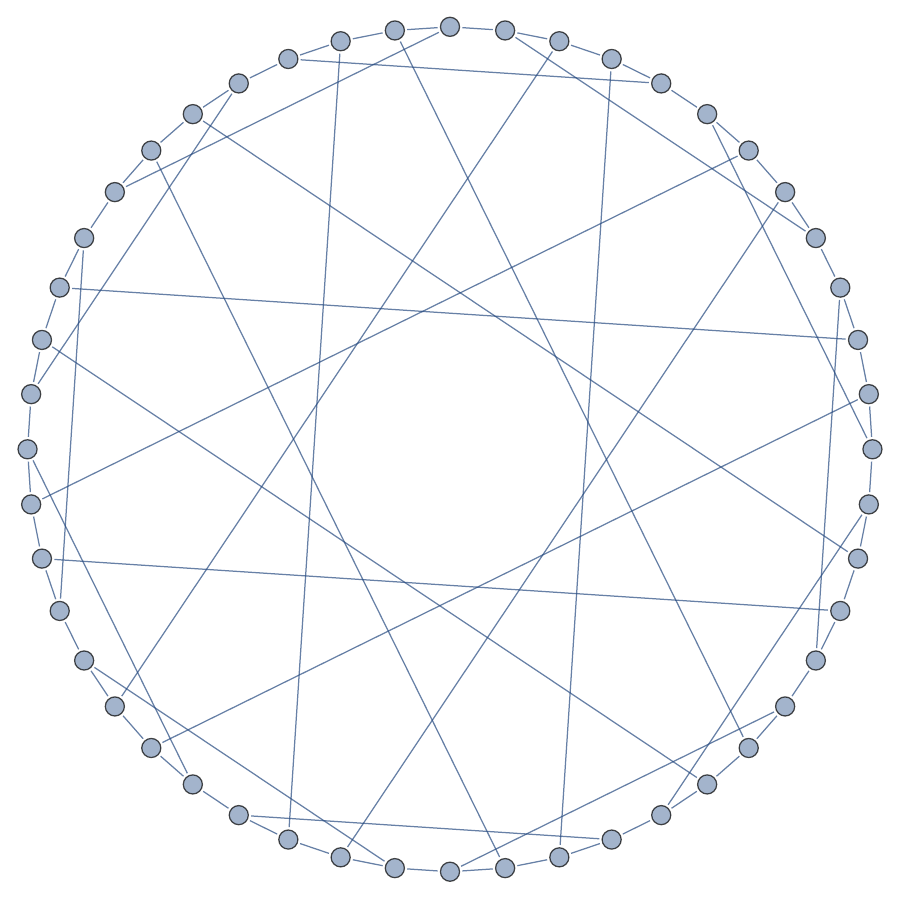}} & {\includegraphics[scale=0.1]{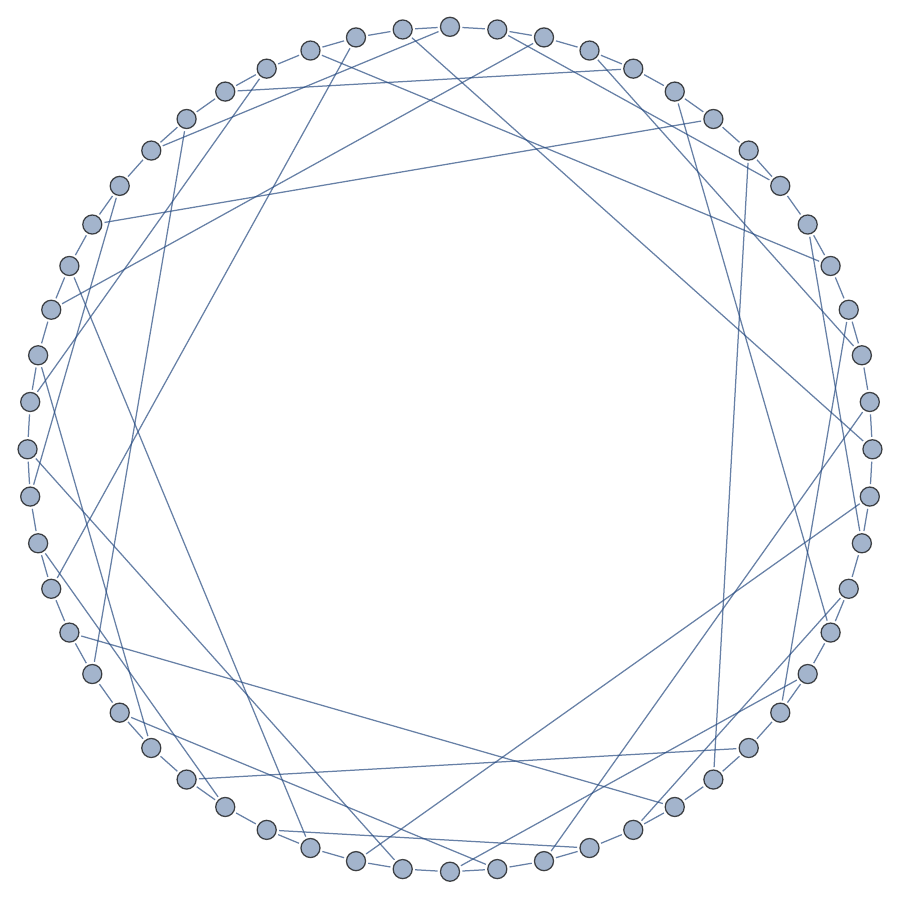}} & {\includegraphics[scale=0.1]{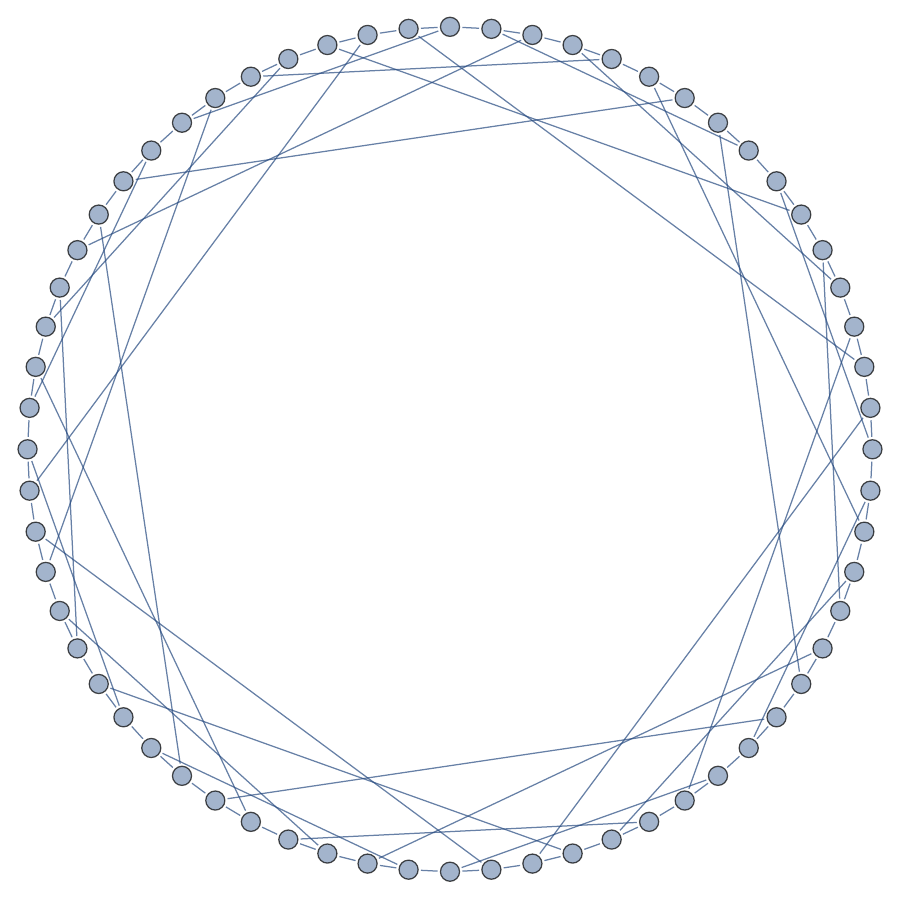}} & $\ldots$ \\
\hline
\noalign{\smallskip}  
D3   &   & 7\ 9\   &  7\ 19 & 7\ 9\  & 7\ 9\  & 7\ 9\   \\
        &   & 13\ 15  &  7\ 19\  & 13\ 15 & 13\ 15 & 13\ 15  \\
& 
 \\
\hline
\end{tabular}
\end{figure}

\begin{figure}[htpb]
\caption{$(3, 8)$ HBGs with symmetry factor $3$ (Non-monotonic)} 
\label{ex_mono_not} 
\centering
\label{table_graph_g8_sym3}
\begin{tabular}{l | l | l | l | l | l | l | c}
\noalign{\smallskip}
\hline
\noalign{\smallskip}  
Order & 30 & 36 & 42 & 48 & 54 & $60 + 6i,$ \\
 &  &  &  &  &  & $i = 0, 1, \ldots$ \\
\hline
\noalign{\smallskip}  
Existence&  $\exists$ & $\neg$ $\exists$ & $\exists$  & $\exists$  & $\exists$  & $\exists$  \\
\hline
\noalign{\smallskip}  
Figure & {\includegraphics[scale=0.1]{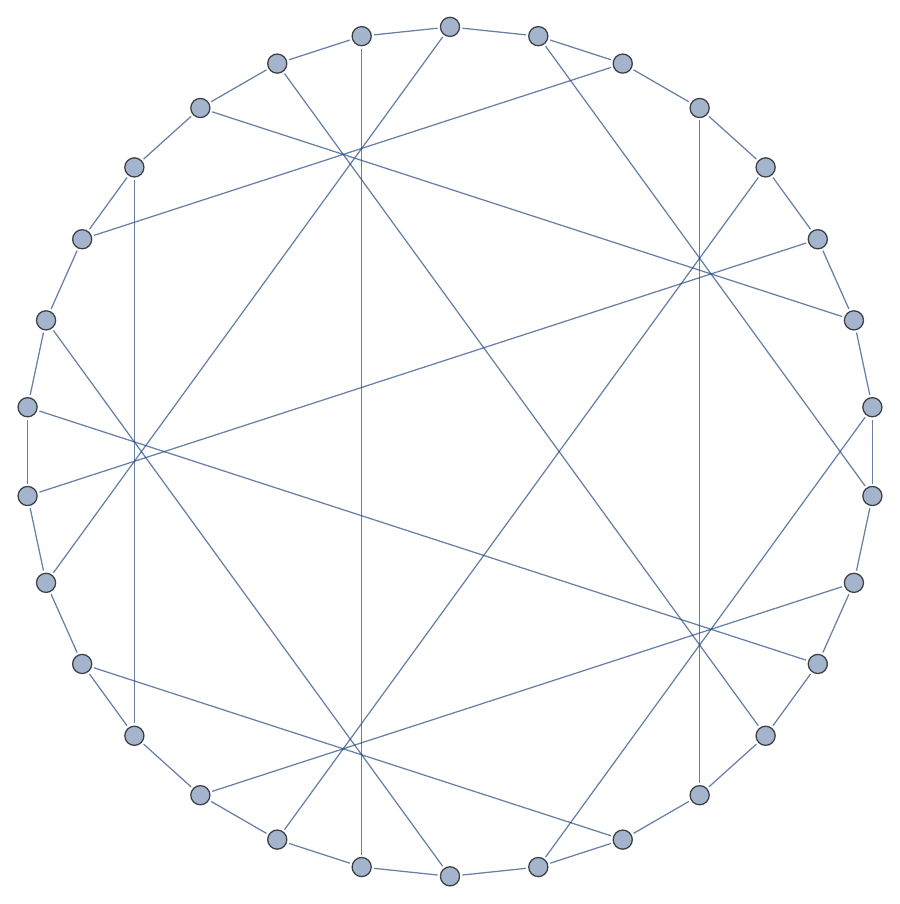}}  &  & {\includegraphics[scale=0.1]{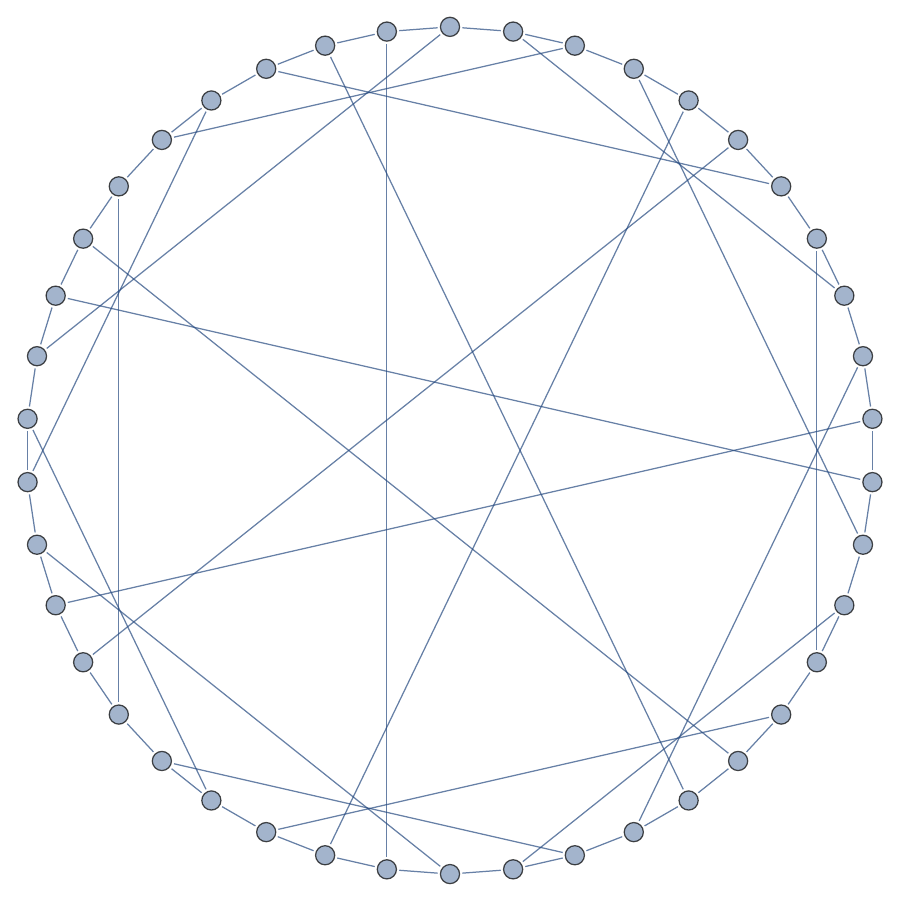}} & {\includegraphics[scale=0.1]{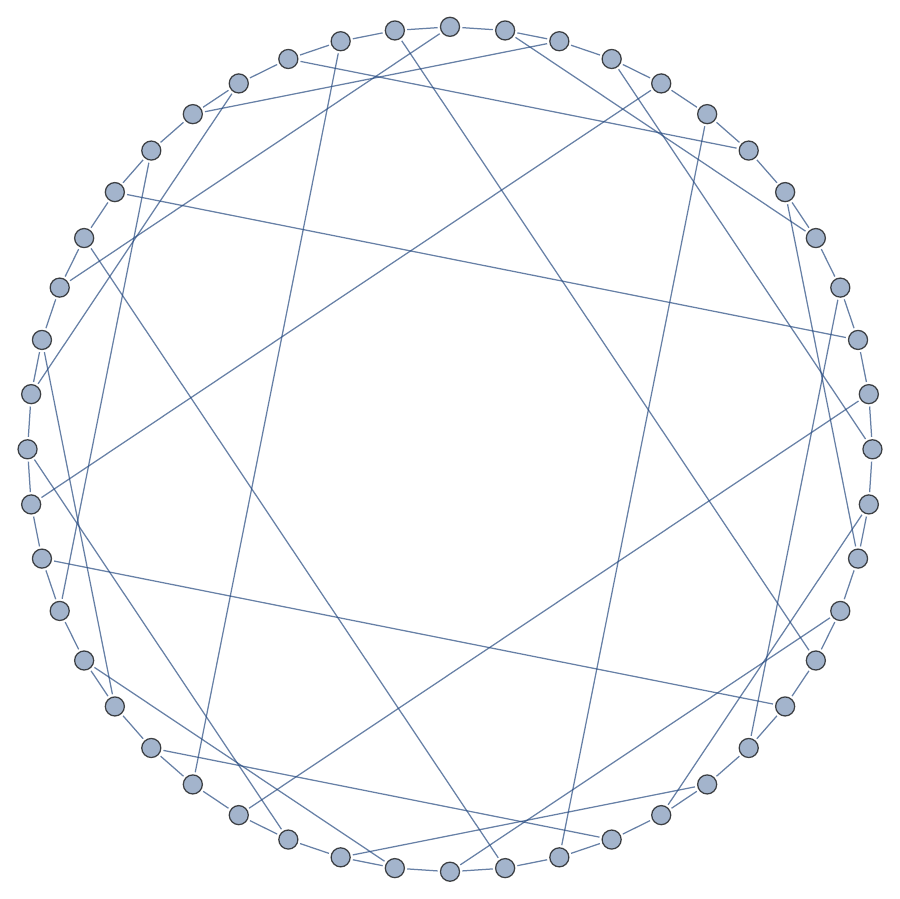}} & {\includegraphics[scale=0.1]{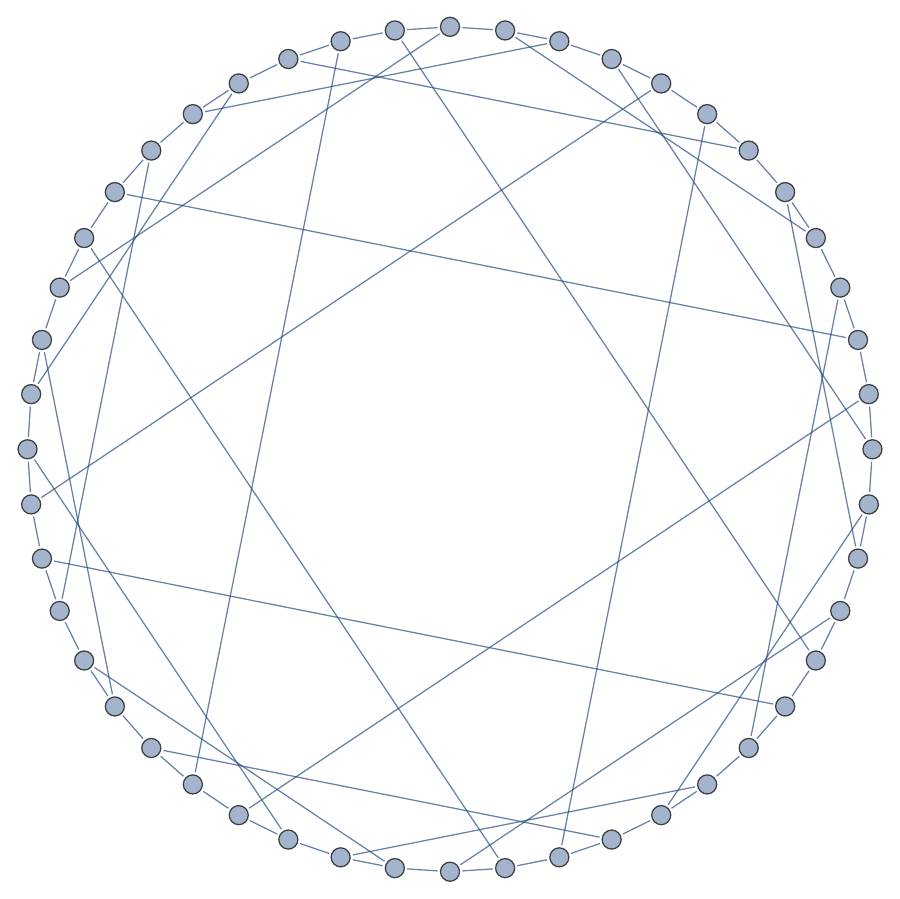}} & $\ldots$ \\
\hline
\noalign{\smallskip}  
D3      & 7\ 9\  &  &  7\ 9\  & 7\ 9\  & 7\ 9\  & 7\ 9\   \\
           & 17  & &  23  & 17 & 29 & 17  \\
& 
 \\
\hline
\end{tabular}
\end{figure}


The minimum observed symmetry factor and lower bound predicted by Theorem \ref{thm_sf_g} for $(3, g)$ HBGs for even values of girth $g$ has been plotted in Table \ref{table_thm_sf_g}.
The minimum observed symmetry factor exactly matches the lower bound predicted by Theorem \ref{thm_sf_g} for $(3, 6)$ HBGs, $(3, 8)$ HBGs, $(3, 10)$ HBGs and $(3, 12)$ HBGs, while the minimum observed symmetry factor is greater than the lower bound predicted By Theorem \ref{thm_sf_g} for $(3, 14)$ HBGs and $(3, 16)$ HBGs. Hence, it is indeed likely that there might exist a $(3, 16)$ HBG with symmetry factor less than $8$. 

\begin{table}
\centering
\caption{Validation of Theorem \ref{thm_sf_g} with experimental data}
\label{table_thm_sf_g}
\begin{tabular}{cccccllllll}
\hline\noalign{\smallskip}
$(3, g)$ & Minimum observed & Lower bound \\
  & symmetry factor & predicted by Theorem \ref{thm_sf_g} \\
\noalign{\smallskip}
\hline
\noalign{\smallskip}
$(3, 6)$ & $1$ & $1$ & 
\\
$(3, 8)$ & $2$ & $2$ 
\\
$(3, 10)$ & $2$ & $2$  
\\
$(3, 12)$ & $3$ & $3$ 
\\
$(3, 14)$ & $4$ & $3$ 
\\
$(3, 16)$ & $8$ & $4$ & \\

\hline
\end{tabular}
\end{table}

\subsection{Completeness and correctness of the catalog}
As pointed out earlier, this catalog \cite{CatalogPaper} does not aim to be complete by listing all HBGs for a given order, since currently only one HBG from a representative class of order, girth and symmetry factor is listed. There are many instances of multiple HBGs of the same order and girth on the catalog \cite{CatalogPaper}, with different symmetry factors. 
From the perspective of correctness of the catalog \cite{CatalogPaper}, order of HBG is correct by construction, girth of the HBG is checked using a computer program. As pointed out earlier, every HBG has a D3 chord index representation. The conclusion about non-existence of a HBG with specified order, girth and symmetry factor is arrived after computer search of the entire space of D3 chord indices corresponding to the order and symmetry factor.

\section{Findings}
\label{sec_findings}
The key findings of this work are as follows. Some of these points have been referred to in earlier sections, but have been listed explicitly here in order to make them clearer. 
A more detailed discussion of Table \ref{table_comp1} has been provided in \cite{OverallPaper2}.
\begin{table}
\caption{Comparison of orders for which $(3, g)$ graphs exist in catalog from \cite{OverallPaper2}}
\label{table_comp1}
\centering
\begin{tabular}{cccccc}
\hline
    $(3, g)$ & Until &Hamiltonian  & Vertex  &  Symmetric   \\
 graphs  & order & bipartite & -transitive &  \\
\noalign{\smallskip}
\hline
\noalign{\smallskip}
$(3, 6)$ & 50 & 19 & 19 & 10   \\
 $(3, 8)$ &90 & 29  & 21  & 6 \\
$(3, 10)$ & 160 & 29 & 15   & 7 \\
$(3, 12)$ & 400& 84 &26  & 16 \\
$(3, 14)$ & 1000& 164 & 35  & 11\\
\hline
\end{tabular}
\end{table}

\begin{enumerate}
\item \textbf{Finding 1}: \textbf{Identification of HBGs as a promising class of cubic graphs} \\
HBGs have been identified, as a promising class of cubic graphs that can lead to a catalog of $(3, g)$ graphs for even girth $g$ with graphs for more orders than other listings, that is also expected to contain a $(3,g)$ graph with minimum order. This is the first known work to list HBGs of specified even girths, with the obtained catalog having graphs for larger number of distinct orders than the previous lists, and includes the $(3,g)$ cage for $g = 6, 8$ and $12$, one $(3, 10)$ cage and the $(3, 14)$ record graph. 
\item \textbf{Finding 2}: \textbf{This catalog has $(3, g)$ graphs for more orders compared to other lists}
\begin{enumerate}
\item Result on existence of a $(3, g)$ HBG for each order for which $(3, g)$ vertex-transitive graphs exist, whether bipartite or non-bipartite for chosen intervals of comparison for even girth until $g = 14$. For comparison with vertex-transitive graphs, we consider $(3, 6)$ graphs until order $50$, $(3, 8)$ graphs until order $90$, $(3, 10)$ graphs until order $160$, $(3, 12)$ graphs until order $400$ and $(3, 14)$ graphs until order $1000$, as shown in Table \ref{table_comp1}.
\item The catalog provides detailed information on existence of $(3,g)$ HBGs for each order, with symmetry factors for which graphs have been found, symmetry factors for which graphs have been proven not to exist, and symmetry factors for which the existence of the graph have been found to be \textit{inconclusive}. The details of this are provided in \cite{OverallPaper2}.
\item Large number of $(3, g)$ graphs outside the vertex-transitive class: There are many orders for which a $(3, g)$ vertex-transitive graph does not exist, for which $(3, g)$ HBG exists on this catalog, as seen in Table \ref{table_comp1}.
\item The catalog of $(3,g)$ HBGs are \textit{exhaustive} with graphs for each order in specified range with proof of non-existence for orders of graphs not listed for $g = 6, 8$ and \textit{partial} for $10, 12, 14$, and $16$ due to some \textit{inconclusive} results for some subclasses.
\end{enumerate}
\item \textbf{Finding 3}: \textbf{Graph representation relevant for this research}
\begin{enumerate}
\item Symmetry factor has been introduced for representing the level of rotational symmetry in cubic Hamiltonian bipartite graphs. Symmetry factor allows decomposing the problem of listing $(3, g)$ HBGs for even girth $g$ for a range of orders into sub-problems of listing $(3, g)$ HBGs for a specified symmetry factor $b$ for a range of orders, and hence allows listing of $(3, g)$ HBGs for more orders.
\item D3 chord index notation has been introduced, which is twice as compact as the LCF notation for representing cubic Hamiltonian bipartite graphs. 
\item The D3 chord index notation can specify an infinite family of HBGs, 
for different orders greater than a threshold value.
\begin{example}
D3 chord index 5 leads to a $(3, 6)$ HBG for all even orders greater than or equal to $14$ as shown in Figure \ref{example_3_6_inf}. This has been proved theoretically and practically verified until order $20008$.
\end{example}
\begin{example}
It is practically observed that D3 chord indices \\$15\ 53\ 73\ 139\ 243\ 267\ 471\ 651$ leads to $(3, 16)$ HBGs for orders $2352 + 16i$ for integers $i \ge 0$, for symmetry factor $8$. In addition, it is also observed that the above mentioned D3 chord indices also lead to $(3, 16)$ HBGs for the following orders \\$1824\ 1840\ 1936\ 2016\ 2032\ 2112\ 2144\ 2160\ 2176\ 2240\ 2256\ 2272$\\ $2288\, 2304\, 2320$.
\end{example} 
\end{enumerate}
\item
\begin{enumerate}
\item \textbf{Finding 4}: \textbf{Extremal graph theory perspective} \\
The following questions from an extremal graph theory perspective have been raised.
\begin{question}
\label{egt_q1}
Given a specified symmetry factor $b$, what is the minimum order for the existence of a $(3, g)$ HBG?
\end{question}
\begin{question}
\label{egt_q2}
“Given a specified even value of girth $g$, what is the minimum value of symmetry factor $b$ for existence of a $(3,g)$ HBG?”
\end{question}
Question \ref{egt_q1} and question \ref{egt_q2} are important in order to understand the relationship between symmetry factor $b$, girth $g$ and order for $(3, g)$ HBGs. 
This research provides lower bounds and upper bounds on minimum order for the existence of a $(3, g)$ Hamiltonian bipartite graph for a specified symmetry factor $b$ for even girth $g$ in Question \ref{egt_q1}, with the lower bound being equal to the upper bound for many cases. For example, for $(3, 14)$ Hamiltonian bipartite graphs the corresponding lower bound and upper bounds are equal for symmetry factors 4, 5 and 6, as shown in Table \ref{table_comp314}.
This research provides a theoretical lower bound and empirical data for the minimum value of symmetry factor $b$ for a $(3, g)$ Hamiltonian bipartite graph for a range of even girth $g$ for Question \ref{egt_q2}. For example, the minimum value of symmetry factor for $(3, 6)$ HBGs is 1, and for $(3, 8)$ HBGs is 2, as shown in Table \ref{table_thm_sf_g}.
\item \textbf{Finding 5}: \textbf{Cage problem perspective} \\
Results on non-existence of $(3,g)$ HBGs for some orders and symmetry factors that show the emptiness of subclasses of graphs are of wider interest. For example, for $(3, 14)$ graphs, the lower bound is 258 and upper bound is 384, the non-existence of $(3, 14)$ HBGs between orders 258 and 384 for symmetry factors 4, 5, 6, and emptiness of many other subclasses as per the non-existence lists, has been shown as discussed in \cite{3_14Paper}.
\end{enumerate}
\end{enumerate}

\section{Limitations and open problems}
\label{sec_conclusion_graph_analysis}
\label{chap_14_16_conclu}
\label{chap_6_8_10_12_conclu}

\subsection{Restrictions to this approach}

Some of the restrictions and limitations of this approach are as follows.
\begin{enumerate}
\item The scope of this approach is restricted to cubic HBGs. 
\item This approach is restricted to even values of girth.
\item Results on existence or emptiness of some subclasses of $(3, g)$ HBGs for some orders, symmetry factor and even girth $g$ are \textit{inconclusive}.
\item The size of the automorphism group for each graph is not computed.
\end{enumerate}

\subsection{Open problems}

The following open problems on $(3, g)$ graphs arise out this research.\\ 
\begin{open problem}
Given an arbitrary even number $2m$ greater than 14, what is the maximum attainable girth of cubic graphs with order $2m$?
\end{open problem}
In general, it might not be very straightforward to find the maximum attainable girth of cubic bipartite graphs for a specified even order. The catalog \cite{CatalogPaper} suggests that the maximum attainable girth for cubic graphs for orders $14 \le 2m \le 28$ is 6, for orders $34 \le 2m \le 90$ is 8, and it is well known that it is 8 for order 30, since $n(3, 8) = 30$.
\begin{open problem} 
\label{conj_chord_in3}
\label{conj_chord_in4}
Let $g \ge 6$ be the maximum attainable girth in the space of cubic bipartite graphs with specified even order $2m$ greater than 14. Does there always exist a $(3, g_{1})$ HBG with order $2m$ for all even values of girth $g_{1}$ satisfying $6 \le g_{1} \le g$?
\end{open problem}
The catalog \cite{CatalogPaper} introduced in this paper suggests that the answer to open problem \ref{conj_chord_in4} might be yes, but it is not currently known whether this is true in general. 
As pointed out earlier, the research of Haythorpe 2014 \cite{Haythorpe2014} suggests that non-Hamiltonian cubic graph of minimum order is always larger than a Hamiltonian cubic graph of minimum order. It is indeed likely that this might be true in general. Hence, open problem \ref{conj_chord_in5} is raised.
\begin{open problem} 
\label{conj_chord_in5}
Given $m \in \mathbb{N}$, if $g$ is the maximum attainable girth in the space of all possible cubic bipartite graphs with order $2m$, then does there exist a $(3, g)$ bipartite graph with order $2m$ that is not Hamiltonian?
\end{open problem}
As pointed out earlier, it is well known that all the known $(3, g)$ cages for even girth $g$ are Hamiltonian. The $(3, 7)$ cage found by McGee 1960 \cite{McGee7} and $(3, 11)$ cage found by Balaban 1973 \cite{Balaban11} are known to be Hamiltonian, and the eighteen $(3, 9)$ cages which were shown to be complete by Brinkmann et al. 1995 \cite{Brinkmann_3_9} have been checked by this author and found to be Hamiltonian as well. Hence, open problem \ref{conj_chord_Hcg} is raised.
\begin{open problem}
\label{conj_chord_Hcg}
Does every $(3, g)$ cage, other than the known exception of the $(3, 5)$ cage, the Petersen graph, have a Hamiltonian cycle?
\end{open problem}
Open problem \ref{Q6_d3} ia raised on size of its automorphism group of a cubic HBG. In general, finding the automorphism group of a graph is known to be a difficult problem. Quoting from Lauri J. et al \cite{Sym_book}, ``The problem of determining whether a graph has a nontrivial automorphism is in NP, since, given such a permutation of the vertices, it is easy to determine in polynomial time that it is an automorphism.''
If we consider the class of circulant graphs, the automorphism group is known in very few cases as pointed out by Morris \cite{JoyMorris}.
\begin{open problem} \textbf{For size of automorphism group} \\
\label{Q6_d3}
Given the D3 chord indices for a cubic HBG, what can one say about the size of its automorphism group?
\end{open problem}

\bibliographystyle{plain} 
\bibliography{vivek}

\begin{thebibliography}{10}

\bibitem{Faradzev}
Faradzev~I. A.
\newblock Constructive enumeration of combinatorial objects.
\newblock {\em Colloques internationaux C.N.R.S. No 260 - Problemes
  Combinatoires et Theorie des Graphes, Orsay}, pages 131--135, 1976.

\bibitem{Goldberg1992}
Goldberg~L. A.
\newblock Efficient algorithms for listing unlabeled graphs.
\newblock {\em Journal Of Algorithms}, 13:128--143, 1992.

\bibitem{Holton85}
Holton~D. A., Manvel B., and McKay~B. D.
\newblock {{H}amiltonian Cubic 3-Connected Cubic Bipartite Planar Graphs}.
\newblock {\em Journal Of Combinatorial Theory, Series B}, 38(3), 1985.

\bibitem{Holton88}
Holton~D. A. and McKay~B. D.
\newblock The smallest non-{H}amiltonian 3-connected cubic planar graphs have
  38 vertices.
\newblock {\em Journal Of Combinatorial Theory, Series B}, 45(3):468--488,
  1988.

\bibitem{Balaban11}
Balaban A.T.
\newblock Trivalent graph of girth nine and eleven and relationships among the
  cages.
\newblock {\em Rev. Roumaine Math.}, 18:1033--1043, 1973.

\bibitem{Brinkmann_3_9}
McKay B.~D. Brinkmann~G. and Saager C.
\newblock The smallest cubic graphs of girth nine.
\newblock {\em Combin. Probab. Comput.}, 5:1--13, 1995.

\bibitem{Brinkmann2011_gen}
McKay B.~D. Brinkmann~G., Goedgebeur~J.
\newblock Generation of cubic graphs.
\newblock {\em Discrete Mathematics and Theoretical Computer Science}, 2011.

\bibitem{Bussemaker1976}
Bussemaker~F. C., Cobeljic S., Cvetkovic~D. M., and Seidel~J. J.
\newblock Computer investigation of cubic graphs.
\newblock Technical report, Technical Report T.H.Report 76-WSK-01, Department
  of Math. Technological University Eindhoven, The Netherlands, 1976.

\bibitem{Read1981}
Read~R. C.
\newblock A survey of graph generation techniques.
\newblock {\em Lecture Notes in Mathematics}, 884:77--89, 1981.

\bibitem{ZS_census}
Powers D.~L. Coxeter~H.S.M., Frucht~R.
\newblock {\em Zero-symmetric graphs : trivalent graphical regular
  representations of groups}.
\newblock Academic Press, 1981.

\bibitem{McKay1987}
McKay~B. D.
\newblock Independent sets in regular graphs of high girth.
\newblock {\em Ars Combinatorica}, 23A:179--185, 1987.

\bibitem{Spence3}
McKay~B. D. and Spence E.
\newblock Classification of regular two-graphs on 36 and 38 vertices.
\newblock {\em Australasian Journal of Combinatorics}, 24:293, 2001.

\bibitem{McKay1986}
McKay~B. D. and Royle~G. F.
\newblock Constructing the cubic graphs on up to 20 vertices.
\newblock {\em Ars Combinatorica}, 21A:129--140, 1986.

\bibitem{105}
McKay~B. D., Myrvold W., and Nadon J.
\newblock Fast backtracking principles applied to find new cages.
\newblock In {\em Proceedings of the ninth annual ACM-SIAM symposium on
  Discrete algorithms, SODA '98, Society for Industrial and Applied Mathematics
  Philadelphia, PA, USA}, pages 188--191, 1998.

\bibitem{Spence2}
Spence E.
\newblock Regular two-graphs on 36 vertices.
\newblock {\em Linear Algebra and its Applications}, pages 226--228, 459--497,
  1995.

\bibitem{Spence1}
Spence E.
\newblock The strongly regular (40,12,2,4) graphs.
\newblock {\em Electronic Journal of Combinatorics}, 7(1), 2000.

\bibitem{HararyPalmer}
Harary F. and Palmer~E. M.
\newblock {\em Graphical Enumeration}.
\newblock Academic Press, 1973.

\bibitem{Brinkmann1992}
Brinkmann G.
\newblock Generating cubic graphs faster than isomorphism checking.
\newblock Technical report, Technical report, SFB 343, Universitat Bielefeld,
  1992.

\bibitem{Brinkmann}
Brinkmann G., McKay~B. D., and Saager C.
\newblock The smallest cubic graphs of girth nine.
\newblock {\em Combin. Probab. Comput.}, 5:1--13, 1995.

\bibitem{Brinkmann_hist2013}
Brinkmann G., Goedgebeur J., and Cleemput N.V.
\newblock The history of the generation of cubic graphs.
\newblock {\em International Journal of Chemical Modeling}, 5(2-3):67--89,
  2013.

\bibitem{HoG}
Brinkmann G., Coolsaet K., Goedgebeur J., and Melot H.
\newblock House of graphs: a database of interesting graphs.
\newblock {\em Discrete Applied Mathematics}, 161(1–2):311--314, 2013.

\bibitem{Brinkmann2013}
Brinkmann G., Cleemput~N. V., and Pisanki T.
\newblock Generation of various classes of trivalent graphs.
\newblock {\em Theoretical Computer Science}, 502:16--29, 2013.

\bibitem{48}
Exoo G.
\newblock Regular graphs of given degree and girth.
\newblock \url{http://ginger.indstate.edu/ge/CAGES}.

\bibitem{Exoo}
Exoo G.
\newblock A small trivalent graph of girth 14.
\newblock {\em The Electronic Journal Of Combinatorics}, 9, 2002.

\bibitem{112}
Exoo G., McKay~B. D., Myrvold~W. J., and Nadon J.
\newblock Computational determination of (3, 11) and (4, 7) cages.
\newblock {\em J. Discrete Algorithms}, 9(2):166--169, 2011.

\bibitem{Jajcaysurvey}
Exoo G. and Jajcay R.
\newblock Dynamic cage survey.
\newblock {\em Electronic Journal of Combinatorics}, 18(DS16), 2011.

\bibitem{Baker1974}
Baker~H. H., Dewdney~A. K., and Szilard~A. L.
\newblock The production of graphs by computer.
\newblock {\em Mathematics of Computation}, 28(127):833--838, 1974.

\bibitem{Heap1972}
B.R. Heap.
\newblock The production of graphs by computer.
\newblock {\em Graph Theory and Computing, Academic Press}, 68:47--62, 1972.

\bibitem{Kagno1946}
Kagno I.
\newblock Linear graphs of degree less than 7 and their groups.
\newblock {\em American Journal of Mathematics}, 68:505--529, 1946.

\bibitem{Vries1889}
DeVries J.
\newblock Over vlakke configuraties waarin elk punt met twee lijnen incidentis.
\newblock {\em Verslagen en Mededeelingen der Koninklijke Akademie voor
  Wetenschappen, Afdeeling Natuurkunde}, 3(6):382--407, 1889.

\bibitem{Vries1891}
DeVries J.
\newblock Sur les configurations planes dont chaque point supporte deux
  droites.
\newblock {\em Rendiconti Circolo Mat. Palermo}, 5:221--226, 1891.

\bibitem{Hoare1983}
Hoare~M. J.
\newblock On the girth of trivalent {C}ayley graphs. graphs and other
  combinatorial topics.
\newblock In {\em Proceedings of the Third Czechoslovak Symposium on Graph
  Theory, Prague 1982, Teubner, Leipzig}, pages 109--114, "" 1983.

\bibitem{Sym_book}
Lauri J. and Scapellato R.
\newblock {\em Topics in Graph Automorphisms and Reconstruction}.
\newblock Cambridge University Press, London Mathematical Society Lecture Note
  Series 432, 2016.

\bibitem{Lederberg}
Lederberg J.
\newblock A system for computer construction, enumeration and notation of
  organic molecules as tree structures and cyclic graphs.
\newblock Technical report, Part 11: Topology of cyclic graphs. Interim Report,
  Stanford, 1965.

\bibitem{JoyMorris}
Morris J.
\newblock Automorphism groups of circulant graphs - a survey.
\newblock {\em draft}.

\bibitem{Spence5}
Coolsaet K., Degraer J., and Spence E.
\newblock The strongly regular (45,12,3,3) graphs.
\newblock {\em Electronic Journal of Combinatorics}, 13(1), 2006.

\bibitem{30}
Wong~P. K.
\newblock Cages a survey.
\newblock {\em Journal of Graph Theory}, 6:1--22, 1982.

\bibitem{Lovasz}
Lovasz L.
\newblock Combinatorial structures and their applications.
\newblock {\em Proc. Calgary Internat. Conf., Calgary, Alberta, 1969}, Problem
  11, Gordon and Breach, New York:243–246, 1970.

\bibitem{Rodriquez}
Rodriquez L.
\newblock Automorphism groups of simple graphs.
\newblock {\em
  \url{https://www.whitman.edu/Documents/Academics/Mathematics/2014/rodriglr.pdf}}.

\bibitem{mconder}
Conder M.
\newblock {Small trivalent graphs of large girth}.
\newblock Technical report, University of Auckland Centre for Discrete
  Mathematics and Theoretical Computer Science Research Report Series, 1997.

\bibitem{Conder_semi}
Conder M., Malnic A., Marusic D., and Potočnik P.
\newblock A census of semisymmetric cubic graphs on up to 768 vertices.
\newblock {\em Journal of Algebraic Combinatorics}, 23:255--294, 2006.

\bibitem{Haythorpe2014}
Haythorpe M.
\newblock Non-{H}amiltonian 3–regular graphs with arbitrary girth.
\newblock {\em Universal Journal of Applied Mathematics}, 2(1):72--78, 2014.

\bibitem{Meringer1999}
Meringer M.
\newblock {Fast Generation of Regular Graphs and Construction of Cages}.
\newblock {\em Journal of Graph Theory}, 30:137--146, 1999.

\bibitem{VTcen2}
Potočnik P., Spiga P., and Verret G.
\newblock Bounding the order of the vertex-stabiliser in 3-valent
  vertex-transitive and 4-valent arc-transitive graphs.
\newblock {\em arXiv:1010.2546v1 [math.CO]}.

\bibitem{VTcen1}
Potočnik P., Spiga P., and Verret G.
\newblock Cubic vertex-transitive graphs on up to 1280 vertices.
\newblock {\em Journal of Symbolic Computation}, 50:465--477, 2013.

\bibitem{LCFpaper}
Frucht R.
\newblock A canonical representation of trivalent {H}amiltonian graphs.
\newblock {\em Journal of Graph Theory}, 1(1):45--60, 1976.

\bibitem{Radziszowski}
Radziszowski.
\newblock {Small Ramsey Numbers}.
\newblock {\em THE ELECTRONIC JOURNAL OF COMBINATORICS}, DS1.14, 2014.

\bibitem{CatalogPaper}
Nittoor~V. S.
\newblock A catalog of $(3, g)$ {H}amiltonian bipartite graphs.
\newblock {\em draft}.

\bibitem{OverallPaper2}
Nittoor~V. S.
\newblock Some properties of catalog of (3, g) hamiltonian bipartite graphs:
  orders, non-existence and infiniteness.
\newblock {\em arXiv:1601.02887 [math.CO]}.

\bibitem{3_14Paper}
Nittoor~V. S.
\newblock Sub-problems of the $(3, 14)$ cage problem and their computer
  analysis.
\newblock {\em arXiv:1612.07683 [math.CO]}.

\bibitem{Sanjmyatav2000}
Sanjmyatav S.
\newblock {Algorithms for generation of cubic graphs}.
\newblock Master's thesis, Australian National University, 2000.

\bibitem{Balaban1967}
Balaban~A. T.
\newblock Valence-isomerism of cyclopolyenes.
\newblock {\em Revue Roumaine de chimie}, 11(12):1097--1116,103 (erratum),
  1967.

\bibitem{32}
Balaban~A. T.
\newblock A trivalent graph of girth ten.
\newblock {\em Journal of Combinatorial Theory}, 12:1--5, 1972.

\bibitem{42}
Tutte~W. T.
\newblock A family of cubical graphs.
\newblock {\em Proceedings of the Cambridge Philosophical Society},
  43:459--474, 1947.

\bibitem{Spence4}
Haemers W. and Spence E.
\newblock The pseudo-geometric graphs for generalised quadrangles of order
  (3,t).
\newblock {\em European Journal of Combinatorics}, 22(6):839--845, 2001.

\bibitem{Robinson1983}
Robinson~R. W. and Wormald~N. C.
\newblock Numbers of cubic graphs.
\newblock {\em Journal of Graph Theory}, 7:463--467, 1983.

\bibitem{McGee7}
McGee W.F.
\newblock A minimal cubic graph of girth seven.
\newblock {\em Canad. Math. Bull.}, 3:149--152, 1960.

\bibitem{FosterC}
Bouwer Z., Chernoff~W. W., Monson B., and Star Z.
\newblock The {F}oster census.
\newblock Technical report, Charles Babbage Research Centre, 1988.

\end{thebibliography}

\end{document}